\documentclass[11pt]{article}
\usepackage[margin = 1in]{geometry}

\usepackage{lipsum}
\usepackage[utf8]{inputenc} 
\usepackage{amsfonts}
\usepackage{graphicx}
\graphicspath{ {images/} }
\usepackage{epstopdf}
\usepackage{dsfont}
\usepackage{mathtools}
\usepackage{empheq}
\usepackage{amsfonts}
\usepackage{amsgen,amsthm,amsmath,amstext,amsbsy,amsopn,amssymb,subfigure,stmaryrd,mathabx,mathrsfs}
\usepackage{comment}
\usepackage[font=footnotesize,labelfont=bf]{caption}
\usepackage{xcolor}

\usepackage{blkarray}

\usepackage{url}
\usepackage{longtable}
\usepackage{mathtools}
\usepackage{multirow}
\usepackage{array}

\usepackage{tikz}

\usepackage[
backend=biber,
natbib=true,
style=alphabetic,
maxbibnames=15,
maxalphanames=10,
minalphanames=6,
doi=false,
isbn=false,
url=false,
eprint=false,
backref=true,
]{biblatex}
\usepackage{enumitem}
\usepackage{hyperref}
\usepackage{booktabs}
\usepackage{algorithm}
\usepackage{algpseudocode}
\usepackage{xcolor}

\usepackage[T1]{fontenc}

\usepackage[nameinlink,capitalise]{cleveref}
\crefformat{cond}{(#2#1#3)}
\crefname{condition}{Condition}{Conditions}
\Crefmultiformat{cond}
  {#2(#1)#3}{ and #2(#1)#3}{, #2(#1)#3}{, and #2(#1)#3}

\usepackage{thm-restate}
\def\N{\mathbb N}
\newcommand{\cB}{\mathcal{B}}
\newcommand{\he}{\mathrm{he}}
\newcommand\numberthis{\addtocounter{equation}{1}\tag{\theequation}}
\newcommand{\STAT}{{\mathtt{STAT}}}

\DeclarePairedDelimiter\floor{\lfloor}{\rfloor}

\ifpdf
  \DeclareGraphicsExtensions{.eps,.pdf,.png,.jpg}
\else
  \DeclareGraphicsExtensions{.eps}
\fi

\newcommand{\1}{{\mathbf{1}}}

\newcommand{\E}{{\mathbb{E}}}

\newcommand{\R}{{\mathbb{R}}}

\newcommand{\Dirac}[1]{\mathsf{D}_{#1}}      

\newcommand{\cS}{\mathcal{S}}
\newcommand{\cA}{\mathcal{A}}

\newcommand{\cN}{\mathcal{N}}
\newcommand{\cP}{\mathcal{P}}

\newcommand{\cF}{\mathcal{F}}
\newcommand{\cX}{\mathcal{X}}

\newcommand{\cD}{\mathcal{D}}
\newcommand{\cE}{\mathcal{E}}

\newcommand{\bbP}{\mathbb{P}}

\renewcommand{\exp}{{\rm{exp}}}
\newcommand{\TV}{{\sf TV}}
\newcommand{\KL}{{\rm{KL}}}

\renewcommand{\P}{{\rm{P}}}

\newcommand{\dimension}{{p}}

\newcommand{\argmin}{\mathop{\rm arg\min}}
\newcommand{\argmax}{\mathop{\rm arg\max}}

\newcommand{\var}{{\textrm {Var}}}

\newcommand{\indi}{{\mathds{1}}}

\newcommand{\wh}{\widehat}
\newcommand{\wt}{\widetilde}

\newtheorem{Theorem}{Theorem}[section]

\newtheorem{Lemma}[Theorem]{Lemma}

\newtheorem{Fact}[Theorem]{Fact}
\newtheorem{Corollary}[Theorem]{Corollary}

\newtheorem{Proposition}[Theorem]{Proposition}

\newtheorem{Model}{Model}

\theoremstyle{definition}
\newtheorem{Remark}[Theorem]{Remark}
\newtheorem{Definition}[Theorem]{Definition}

\DeclareMathAlphabet\mathbfcal{OMS}{cmsy}{b}{n}

\newcommand{\red}{\color{red}}

\newcommand{\modelref}[1]{%
  \csname modelref@#1\endcsname
}

\expandafter\newcommand\csname modelref@model:huber\endcsname{%
  \Cref{model:huber} (Huber)%
}
\expandafter\newcommand\csname modelref@model:adaptive\endcsname{%
  \Cref{model:adaptive} (Adaptive)%
}
\expandafter\newcommand\csname modelref@model:obv-1\endcsname{%
  \Cref{model:obv-1} (Oblivious I)%
}
\expandafter\newcommand\csname modelref@model:obv-2\endcsname{%
  \Cref{model:obv-2} (Oblivious II)%
}
\expandafter\newcommand\csname modelref@model:non-uni\endcsname{%
  \Cref{model:non-uni} (Non-Uniform)%
}
\expandafter\newcommand\csname modelref@model:hub-bdd-marg\endcsname{%
  \Cref{model:hub-bdd-marg} (Huber with Bounded Marginal Likelihood)%
}
\expandafter\newcommand\csname modelref@model:TV\endcsname{%
  \Cref{model:TV} (TV)%
}

\def\colorful{1}
\ifnum\colorful=1
\usepackage[normalem]{ulem}

\newcommand{\apnote}[1]{\footnote{{\bf [Ankit: {#1}\bf ]}}}
\newcommand{\dnote}[1]{\footnote{{\bf [Dong: {#1}\bf ]}}}
\newcommand{\light}[1]{ {\grey #1}} 
\newcommand{\todo}[1]{{\textbf{ [{\red{Todo}}: {#1}]}}}
\else

\newcommand{\apnote}[1]{}
\newcommand{\dnote}[1]{}
\newcommand{\todo}[1]{}
\newcommand{\light}[1]{} 
\fi

\hypersetup{
    colorlinks=true,
    linkcolor=blue,
    filecolor=blue,      
    urlcolor=cyan,
    citecolor=blue
}
\newcommand\blfootnotea[1]{%
  \begingroup
  \renewcommand\thefootnote{}\footnote{#1}%
  \endgroup
}
\makeatletter
\newcommand*{\rom}[1]{\expandafter\@slowromancap\romannumeral #1@}
\makeatother

\allowdisplaybreaks

\begin{document}
	\title{Robust Regression with Adaptive Contamination in Response: Optimal Rates and Computational Barriers\blfootnotea{Authors are listed in alphabetical order.}}

    \author{
Ilias Diakonikolas\thanks{Supported by NSF Medium Award CCF-2107079 and an H.I. Romnes Faculty Fellowship.}\\
University of Wisconsin-Madison \\
{\tt ilias@cs.wisc.edu} \\
\and 
Chao Gao\thanks{Supported by NSF Grants ECCS-2216912 and DMS-2310769 and an Alfred Sloan fellowship.}\\
University of Chicago \\
{\tt chaogao@uchicago.edu}
\and 
Daniel M.\ Kane\thanks{Supported by NSF Medium Award CCF-2107547.}\\
University of California, San Diego \\
{\tt dakane@ucsd.edu} \\
\and
Ankit Pensia\\
Carnegie Mellon University\\
{\tt ankitp@cmu.edu} \\
\and
Dong Xie\\
University of Chicago \\
{\tt dongxie@uchicago.edu}
}

	\date{}
	\maketitle

\begin{abstract}
We study robust regression under a contamination model in which covariates are clean while the responses 
may be corrupted in an adaptive manner. Unlike the classical Huber's contamination model, where both 
covariates and responses may be contaminated and consistent estimation is impossible when the 
contamination proportion is a non-vanishing constant, it turns out that the clean-covariate setting 
admits strictly improved statistical guarantees. Specifically, we show that the additional information in the clean 
covariates can be carefully exploited to construct an estimator that achieves a better estimation rate 
than that attainable under Huber contamination. In contrast to the Huber model, this improved rate 
implies consistency even when the contamination is a constant. A matching minimax lower bound is 
established using Fano's inequality together with the construction of contamination processes that match 
$m> 2$ distributions simultaneously, extending the previous two-point lower bound argument in Huber's 
setting. Despite the improvement over the Huber model from an information-theoretic perspective, 
we provide formal evidence---in the form of Statistical Query and Low-Degree Polynomial lower bounds---that the problem exhibits strong information-computation gaps. Our results strongly suggest that 
the information-theoretic improvements cannot be achieved by polynomial-time algorithms, 
revealing a fundamental gap between information-theoretic and 
computational limits in robust regression with clean covariates.
\end{abstract}


\thispagestyle{empty}

\newpage

\setcounter{page}{1}

\begin{sloppypar}
\section{Introduction} \label{sec: introduction}

Robust regression is a central problem in statistics. A canonical setting for robust regression considers data $\{(X_i,y_i)\}_{i=1}^n$ generated according to the following model.
\begin{Model}[Huber Contamination]\label{model:huber}
Let the contamination rate $\epsilon \in [0,1/2)$ and the dimension $\dimension \in \N$.
The pairs $\{(X_i,y_i)\}_{i=1}^n$ are independently drawn from
\begin{equation}
(X_i,y_i)\sim (1-\epsilon)P_{\beta,\sigma}+\epsilon Q, \label{eq:huber-x-y}
\end{equation}
where $Q$ is some arbitrary distribution, and $P_{\beta,\sigma}$ stands for the Gaussian linear model in $p$ dimensions whose sampling process is given by
\begin{equation}
(X_i,y_i)\sim P_{\beta,\sigma}\qquad\Longleftrightarrow\qquad X_i\sim \mathcal{N}(0,I_\dimension)\quad\text{and}\quad y_i\mid X_i\sim \mathcal{N}(X_i^\top\beta,\sigma^2). \label{eq:glm}
\end{equation}
\end{Model}
The data generating process (\ref{eq:huber-x-y}) is known as Huber's contamination model \citep{huber1964robust}. Roughly speaking, the data set $\{(X_i,y_i)\}_{i=1}^n$ contains an $\epsilon$ fraction of outliers under the setting (\ref{eq:huber-x-y}), and an outlier pair $(X_i,y_i)$ takes arbitrary values for both covariates and response.
The restriction of $\epsilon <1/2$ is information-theoretically necessary to get bounded error.
Optimal estimation of the regression vector $\beta$ under \Cref{model:huber} has been well studied in the literature.
When $P_{\beta,\sigma}$ is the Gaussian linear model (\ref{eq:glm}), a rate-optimal robust estimator can be constructed by maximizing the regression depth \citep{rousseeuw1999regression}. It was shown by \cite{Gao20} that maximizing the regression depth achieves the error rate
\begin{equation}
\|\wh{\beta}-\beta\|_2=O_{\mathbb{P}}\left(\sigma\left(\sqrt{\frac{\dimension}{n}}+\epsilon\right)\right), \label{eq:huber-xy-optimal}
\end{equation}
whenever $\epsilon < c < 1/2$, and this rate is information-theoretical optimal among all estimators.
However, maximizing the regression depth is computationally intractable. Polynomial-time algorithms that achieve the (nearly)-optimal rate have been proposed and analyzed by \cite{BalDLS17,DiaKS19,CheATJFB20,pensia2025robust,DiaKPP23-huber-optimal}. 

An important feature of \Cref{model:huber} is the impossibility of consistent estimation when the contamination proportion $\epsilon$ is a non-vanishing constant, due to the presence of the second term $\sigma\epsilon$ in the optimal error rate (\ref{eq:huber-xy-optimal}). One may wonder if a modification
of \Cref{model:huber} leads to consistent estimation even for a constant $\epsilon$. In this paper, we study a natural variation of \Cref{model:huber}, which is defined below.
\begin{Model}[Adaptive Contamination in Responses]\label{model:adaptive}
The pairs $\{(X_i,y_i)\}_{i=1}^n$ are independently drawn from the following process,
\begin{eqnarray}
\label{eq:reg-design} X_i&\sim& \mathcal{N}(0,I_\dimension), \\
\label{eq:reg-huber} y_i\mid X_i &\sim& (1-\epsilon)\mathcal{N}(X_i^\top\beta,\sigma^2) + \epsilon Q_{X_i},
\end{eqnarray}
where $Q_{X_i}$ is some arbitrary conditional distribution depending on $X_i$.
\end{Model}

 When $\epsilon=0$, the setting with (\ref{eq:reg-design}) and (\ref{eq:reg-huber}) is reduced to the Gaussian linear model (\ref{eq:glm}). Compared with (\ref{eq:huber-x-y}) where contamination applies to both $X_i$'s and $y_i$'s, now we only allow the response variables to be contaminated. According to (\ref{eq:reg-huber}), an outlier $y_i$ is drawn from an arbitrary distribution that may depend on the value of $X_i$. This is known as the \textit{adaptive} contamination in the literature, whereas a contamination distribution independent of the covariate is coined as an \textit{oblivious} one. The goal of our paper is to investigate the following questions: 1) In the setting of \Cref{model:adaptive}, is it possible to achieve a strictly better estimation rate than (\ref{eq:huber-xy-optimal})? 2) If the answer to the last question is yes, can we achieve the better estimation rate using a polynomial-time algorithm?

Robust regression with clean covariates has been considered in the literature \citep{she2011outlier,nguyen2012robust,foygel2014corrupted,dalalyan2019outlier,chinot2020erm} in forms that are closely related to \Cref{model:adaptive}. In particular, compared with \Cref{model:huber}, the setting of clean covariates allows straightforward polynomial-time algorithm such as M-estimators to work. Indeed, \cite{dalalyan2019outlier,chinot2020erm} show that a certain M-estimator achieves the rate (\ref{eq:huber-xy-optimal}). However, the results of \cite{dalalyan2019outlier,chinot2020erm} suggest that there is no information-theoretic gain with the additional assumption on the covariates (i.e., the rate (\ref{eq:huber-xy-optimal}) is optimal under \Cref{model:adaptive}). 

In this paper, we show that the clean-covariate assumption in \Cref{model:adaptive} can be leveraged to obtain a strictly better estimation rate than (\ref{eq:huber-xy-optimal}).
Our first main result is the construction of an estimator achieving the following error rate,
\begin{equation}
\|\wh{\beta}-\beta\|_2=O_{\mathbb{P}}\left(\sigma\left(\sqrt{\frac{\dimension}{n}}+\frac{\epsilon}{\sqrt{\log(n\epsilon^2/\dimension+e)}}\right)\right) \;, \label{eq:opt-rate}
\end{equation}
which is also shown to be minimax optimal under \Cref{model:adaptive}.
When $\epsilon \leq \sqrt{\frac{\dimension}{n}}$, both (\ref{eq:huber-xy-optimal}) and (\ref{eq:opt-rate}) are $\sigma\sqrt{\frac{\dimension}{n}}$. On the other hand when $\epsilon\geq \sqrt{\frac{\dimension}{n}}$, the rate (\ref{eq:opt-rate}) becomes $\frac{\sigma\epsilon}{\sqrt{\log(n\epsilon^2/\dimension+e)}}$, compared with the $\sigma\epsilon$ of (\ref{eq:huber-xy-optimal}). In particular, given a constant $\epsilon$, consistency is achieved by (\ref{eq:opt-rate}) whenever $n/p \rightarrow\infty$. The estimator achieving the rate (\ref{eq:opt-rate}) relies on the information at the tail of the design covariates. In other words, unlike the usual approach in robust statistics that trims away large data points, our estimator precisely takes advantage of the additional information associated with design points that have larger projections. This information is only available in the clean covariate setting. Intuitively, large values of $X_i$ mitigate the effect of contamination on the corresponding response $y_i$. This phenomenon holds beyond the setting of Gaussian design (\ref{eq:reg-design}). In a setting where $X_i$'s are generated from a more general class of distributions, we show that the optimal estimation rate critically depends on the tail of the design distribution. The heavier the tail, the faster the rate is.

Since the rate-optimal estimator requires a search over all univariate projections of the covariates, it is computationally infeasible. 
When it comes to polynomial-time algorithms, we establish a Statistical Query (SQ) lower bound \citep{kearns1998efficient,FelGRVX17} showing that any SQ algorithm 
achieving a faster rate than $\sigma\epsilon$ uses either exponentially many queries 
or a single query with an exponentially small tolerance. 
The exponentially small tolerance can be interpreted as exponentially many samples, 
which is still computationally infeasible to process. 
Our SQ lower bound thus rules out the possibility to improve the rate $O(\sigma\epsilon)$, 
for a broad class of polynomial-time algorithms. 
In addition, we also establish a similar computational barrier 
in the Low-Degree Polynomial (LDP) framework \citep{Hopkins-thesis,KunWB19,Wein25-survey} 
by using the connection between the SQ and the low-degree settings \cite{BreBHLS21}. 
To summarize, while there is an information-theoretic separation 
between the robust regression models with and without contamination on the covariates, 
such a separation between \Cref{model:huber} and \Cref{model:adaptive} 
does not hold from a computational perspective. That is, one 
cannot take advantage of the additional structure of the problem 
within a realistic computational budget.



\subsection{Related Work}

We start by summarizing the literature on \Cref{model:huber}. As already mentioned earlier, work on robust statistics~\cite{Gao20} obtained sample-efficient (but computationally inefficient) robust estimators in Huber's contamination model. 
The work of~\cite{BalDLS17} studied (sparse) linear regression
in Huber's contamination model: they observe that robust linear regression
can be reduced to robust mean estimation,
leading to an algorithm whose error guarantee scales with $\|\beta\|_2$. 
\cite{DiaKS19} gave the first sample and computationally efficient algorithm for \Cref{model:huber} with near minimax optimal error guarantees. 
A number of contemporaneous works~\cite{KliKM18,DiaKKLSS19,PraSBR20}
developed robust algorithms for linear regression
under weaker distributional assumptions. The aforementioned algorithmic works succeed even in a stronger contamination model, where the adversary is allowed to adaptively add outliers 
and remove inliers. Focusing on Huber's model in particular,~\cite{DiaKPP23-huber-optimal}
gave a sample near-optimal algorithm with optimal error that runs in near-linear time. 
For additional related work on \Cref{model:huber}, the reader is referred to~\cite{DiaKan22-book}.

In contrast to the vast literature on \Cref{model:huber}, there is no systematic study of \Cref{model:adaptive} to the best of our knowledge. It is worth mentioning \cite{chinot2020erm} that introduced \Cref{model:adaptive} as a hard instance in the lower bound construction of robust regression with clean covariates. In particular, the paper \cite{chinot2020erm} argued that the additional assumption of clean covariates does not make robust regression easier by claiming a minimax lower bound of order $\sigma\left(\sqrt{\frac{\dimension}{n}}+\epsilon\right)$ for \Cref{model:adaptive}. However, by constructing an estimator achieving a faster rate (\ref{eq:opt-rate}), our result indicates that lower bound proof in \cite{chinot2020erm} is incorrect.

On the other hand, various other settings of robust regression with clean covariates have been considered in the literature. The most popular one is the form of linear model with additive outliers \citep{she2011outlier,nguyen2012robust,foygel2014corrupted},
\begin{equation}
y_i= X_i^\top\beta + z_i+\gamma_i, \label{eq:reg-additive-cont}
\end{equation}
where $X_i\sim \mathcal{N}(0,I_\dimension)$, $z_i\sim \mathcal{N}(0,\sigma^2)$, and $\Gamma=(\gamma_1,\cdots,\gamma_n)^\top$ is an outlier vector assumed to be sparse. In fact, when the outlier vector $\Gamma$ is allowed to depend on the covariates, \Cref{model:adaptive} can be written as a special case of the setting (\ref{eq:reg-additive-cont}) with $\epsilon n$ corresponding to the sparsity of $\Gamma$. In \cite{dalalyan2019outlier}, it is proved that the classical M-estimator (e.g. Huber regression) achieves the rate $\sigma\left(\sqrt{\frac{\dimension}{n}}+\epsilon\right)$ up to a logarithmic dimension factor under (\ref{eq:reg-additive-cont}). We will show in \Cref{sec:models} that the rate $\sigma\left(\sqrt{\frac{\dimension}{n}}+\epsilon\right)$ is also a lower bound under (\ref{eq:reg-additive-cont}). In other words, the setting (\ref{eq:reg-additive-cont}) is strictly harder than \Cref{model:adaptive}, and consistency is impossible for a constant $\epsilon$ under (\ref{eq:reg-additive-cont}) just like Huber contamination (\Cref{model:huber}).

Another related model of robust regression with clean covariates deals with oblivious contamination and has been thoroughly studied in the literature \citep{Tsakonas14, JaiTK14,BhatiaJK15,BhatiaJKK17, 
SugBRJ19,pesme2020online, d2021consistent,Steurer21Outliers}. Its canonical setting is given by the linear model $y_i=X_i^\top\beta+z_i$, where
\begin{equation}
X_i\sim \mathcal{N}(0,I_\dimension),\qquad\text{ and }\qquad z_i\sim (1-\epsilon)\mathcal{N}(0,\sigma^2)+\epsilon Q, \label{eq:oblivious}
\end{equation}
and $z_i$ is independent of $X_i$. We note that once the $Q$ in (\ref{eq:oblivious}) is allowed to depend on $X_i$, this recovers (\ref{eq:reg-design})-(\ref{eq:reg-huber}) with adaptive contamination. The minimax rate of estimating the regression coefficients under $\ell_2$ norm is $\sigma\sqrt{\frac{\dimension}{n(1-\epsilon)^2}}$ whenever $\sqrt{\frac{\dimension}{n(1-\epsilon)^2}}$ is sufficiently small \citep{d2021consistent,Steurer21Outliers}, meaning that consistency is possible even when $\epsilon\rightarrow 1$. Compared with the more general formulation in (\ref{eq:reg-huber}), the oblivious contamination model in the form of (\ref{eq:oblivious}) has quite restricted implications in practice. We refer the reader to \Cref{sec:models} for a detailed discussion on the different contamination models.

In many settings of interest, all known statistically optimal estimators 
require exhaustive search to compute, 
while all known computationally efficient algorithms  
achieve weaker guarantees 
than the information-theoretic optimum. 
An information–computation (IC) tradeoff captures precisely this situation, 
where no computationally efficient algorithm for a statistical task
can achieve the information-theoretic limit. 
A successful approach in the literature to providing rigorous evidence of IC tradeoffs 
involves showing unconditional lower bounds within broad (yet restricted)
computational models---including, Low-Degree Polynomials (LDP) and Statistical Queries (SQ).  
Such models capture a wide range algorithmic techniques for statistical tasks, 
and corresponding lower bounds shed light on the structural reasons 
behind observed IC tradeoffs. In this vein,
we study the computational complexity of robust estimation under \Cref{model:adaptive}.
At a high-level, we give formal evidence that achieving the information-theoretically optimal 
error rate in high dimensions may not be possible by any computationally-efficient algorithm.
We briefly discuss the existence of such information-computation gaps in the context of robust linear regression.
As mentioned earlier, such gaps do not exist under \Cref{model:huber}\footnote{When the covariates have an unknown covariance $0.5 I\preceq\Sigma \preceq I$, such gaps do exist~\cite{DiaKS19}.}.
Our work adds to the growing literature that shows that information-computation gaps appear even when the covariates are clean, but the responses are corrupted~\cite{DDKW23a,DDKW23b,DiaGKLP25-oblivious}.
Our work differs from earlier work in two key aspects: 
(i) the contamination model in the earlier works is oblivious, as opposed to adaptive; 
and (ii) these works show a polynomial gap (at most quadratic) between the two rates, 
we establish a super-polynomial gap.

\subsection{Paper Organization}

The rest of the paper first starts with the analysis of the univarite setting in \Cref{sec:uni}. Our main results in high-dimensional setting will be given in  \Cref{sec:main}. The computational hardness of the problem will be established in \Cref{sec:SQ}. In \Cref{sec:disc}, we present some additional discussion on related contamination models and effect of the covariate distribution on the minimax rate. Finally, all technical proofs will be presented in \Cref{sec:pf}.

\subsection{Notation}

We use $\cS^{\dimension-1}$ to denote the set of all unit vectors in $\R^\dimension$.
For a positive natural number $n$, we use the notation $[n]$ and $[0:n]$ to denote the sets $\{1,\dots,n\}$ and $\{0,\dots,n\}$, respectively
All the logarithms would be base $e$.
For an $x \in \R$, we use $x_+$ to denote the positive part of $x$, i.e., $\max(x,0)$.
For a distribution $A$, we sometimes abuse notation by also using $A(\cdot)$ for its probability density function.
For an $x\in \R^\dimension$, $\mu \in \R^\dimension$, and a positive semidefinite matrix (PSD) $\Sigma \in \R^{\dimension \times \dimension}$, we use $\phi(x;\mu, \Sigma)$ to denote the pdf of $\cN(\mu,\Sigma)$ at $x$. 
When the dimension $\dimension$ is clear from the context, we simply write $\phi(x)$ for the pdf of the standard Gaussian $\cN(0,I_\dimension)$.
For a unit vector $v \in \R^\dimension$ and $x \in \R^\dimension$,
we use $\phi_{v^\perp}(x)$ to denote $\exp\left(-\|x - (v^\top X)v\|_2^2/2\right)/(2\pi)^{(\dimension-1)/2}$.
For two distributions $A$ and $B$, we use $A \otimes B$ to denote the product distribution.
For two distributions $P$ and $Q$, the quantities $\TV(P,Q)$ and $\KL(P \| Q)$ denote the TV distance between $P$ and $Q$ and the KL divergence of $P$ from $Q$, respectively.

As the focus of our work is on the (minimax and computational) error rates, we use the standard asymptotic notation $O, \Omega, \Theta$ and $\Theta$ to hide absolute constants that do not depend on other parameters. 
We also use the notation $\lesssim$ to hide absolute constants in the relationship.
For an $x > 0$, we use $\mathrm{poly}(x)$ to denote an expression that is at most a polynomial function of $x$.
For a sequence of random variables $(X_n)_{n \geq 1}$ and real numbers $(a_n)_{n \geq 1}$, we use the notation $X_n= O_{\mathbb{P}}(a_n)$ to mean that, for every $\epsilon>0$, there exist a positive number $M$ and a natural number $n_0 \in \N$ such that for all $n \geq n_0$, $\P(|X_n| \geq M a_n) \leq \epsilon$. 
Finally, when we mention the \emph{error rate} of a problem or an estimator, we hide multiplicative constants.

\section{Prologue: The Univariate Setting}\label{sec:uni}

We will first discuss the univariate setting with $\dimension=1$ and convince the readers that consistent estimation is possible under \Cref{model:adaptive} even when the contamination proportion $\epsilon$ does not vanish. Let us start with a simple median regression estimator,
\begin{equation}
\wh{\beta} = \argmin_{\widetilde{\beta}  \in \R}\frac{1}{n}\sum_{i=1}^n\left|y_i-\widetilde{\beta} X_i\right|. \label{eq:median-reg1d}
\end{equation}
It is not hard to show that the error rate of (\ref{eq:median-reg1d}) is given by
\begin{equation}
|\wh{\beta}-\beta| = O_{\mathbb{P}}\left(\sigma\left(\frac{1}{\sqrt{n}}+\epsilon\right)\right). \label{eq:mreg-rate1d}
\end{equation}
See \Cref{thm:md-reg-hd} for a formal result of (\ref{eq:mreg-rate1d}) with general dimension.
Since the second term of the error rate is $\sigma\epsilon$, median regression is inconsistent unless $\epsilon$ tends to $0$. A key observation here is that the rate (\ref{eq:mreg-rate1d}) also holds under the following data generating process
\begin{eqnarray}
\label{eq:reg-design=} X_i&\sim& \mathcal{N}(0,\sigma^{-2}), \\
\label{eq:reg-huber=} y_i\mid X_i &\sim& (1-\epsilon)\mathcal{N}(\beta X_i,1) + \epsilon Q_{X_i},
\end{eqnarray}
since the setting with (\ref{eq:reg-design=})-(\ref{eq:reg-huber=}) is equivalent to (\ref{eq:reg-design})-(\ref{eq:reg-huber}) when $\dimension=1$ by scaling the data with $\sigma$. More specifically, with $\{(y_i,X_i)\}_{i=1}^n$ sampled from (\ref{eq:reg-design})-(\ref{eq:reg-huber}), we can regard $\{(y_i/\sigma,X_i/\sigma)\}_{i=1}^n$ as sampled from (\ref{eq:reg-design=})-(\ref{eq:reg-huber=}) with some different $Q_{X_i}$. In other words, one expects that the error of median regression is inversely proportional to the magnitude of the covariate. This suggests a new strategy of conditioning on covariates with large magnitude. Thanks to the lack of contamination on covariates, one can always select $X_i$'s that are large first and then apply the median regression.

Inspired by the above discussion, we consider the following truncated median regression estimator,
\begin{equation}
\wh{\beta} = \argmin_{\widetilde{\beta}  \in \R}\frac{1}{n}\sum_{i=1}^n\left|y_i-\widetilde{\beta} X_i\right|\indi\{|X_i|\geq t\}. \label{eq:t-med-reg-1d}
\end{equation}
The truncated median regression is equivalent to applying (\ref{eq:median-reg1d}) with the subset of data indexed by $\{i\in[n]: |X_i|\geq t\}$. In view of (\ref{eq:mreg-rate1d}), one would guess that the error rate of (\ref{eq:t-med-reg-1d}) should be
\begin{equation}
O_{\mathbb{P}}\left(\frac{\sigma}{t}\left(\frac{1}{\sqrt{n(t)}}+\epsilon\right)\right), \label{eq:guess-rate}
\end{equation}
where $n(t)$ stands for the cardinality of $\{i\in[n]: |X_i|\geq t\}$, which is interpreted as the effective sample size. As $t$ increases, the larger magnitude of the covariates implies a smaller $\frac{\sigma\epsilon}{t}$ in the second term of (\ref{eq:guess-rate}). However, the first term $\frac{\sigma}{t\sqrt{n(t)}}$ may get larger because of the smaller effective sample size $n(t)$. To achieve consistency with a constant $\epsilon$, one can choose $t$ such that both $t\rightarrow \infty$ and $t\sqrt{n(t)}\rightarrow\infty$ hold as $n\rightarrow\infty$. For example, with the Gaussian design, by setting $t=\sqrt{\log n}$, we have $n(t)\gtrsim \sqrt{n}$, and the rate (\ref{eq:guess-rate}) is at most $O_{\mathbb{P}}\left(\frac{\sigma}{\sqrt{\log n}}\right)$ even when $\epsilon=0.1$. For a general $\epsilon$ which may be a vanishing function of $n$, the truncation level $t$ should be chosen to minimize the bound (\ref{eq:guess-rate}). This intuition is established as a non-asymptotic error bound in the following theorem.

\begin{Theorem}\label{thm:1d}
Consider data generated from \Cref{model:adaptive} with $\dimension=1$ and the estimator (\ref{eq:t-med-reg-1d}) for some $t\in[0,\sqrt{0.9\log n}]$.
For any $\alpha\in(0,1)$, there exist $C,c>0$ such that whenever $\frac{1}{\sqrt{n}}+\epsilon \leq c$, the estimator (\ref{eq:t-med-reg-1d}) satisfies
$$|\wh{\beta}-\beta|\leq C\frac{\sigma}{t} \left(\frac{1}{\sqrt{n}e^{-t^2/2}}+\epsilon\right),$$
with probability at least $1-\alpha$. Thus, by taking $t=\sqrt{\frac{1}{2}\log(n\epsilon^2+e)}$, we achieve the error rate $\sigma\left(\frac{1}{\sqrt{n}}+\frac{\epsilon}{\sqrt{\log(n\epsilon^2+e)}}\right)$.
\end{Theorem}

In view of the lower bound (\Cref{thm:info-lower} in \Cref{sec:lower}), the estimator (\ref{eq:t-med-reg-1d}) achieves the minimax rate of the problem with an appropriate truncation level.

Motivated by the optimality of the truncated M-estimator in the univariate setting, it is tempting to write down the following extension in high dimension,
\begin{equation}
\wh{\beta} = \argmin_{\widetilde{\beta}\in\mathbb{R}^p}\frac{1}{n}\sum_{i=1}^n\left|y_i-X_i^\top\widetilde{\beta}\right|\indi\{\|X_i\|_2\geq t\}. \label{eq:t-med-reg-hd}
\end{equation}
Unfortunately, the same idea no longer works when $\dimension$ is large. To see why (\ref{eq:t-med-reg-hd}) fails, consider the effective sample size
$$n(t)=\sum_{i=1}^n\indi\{\|X_i\|_2\geq t\}.$$
Since $\|X_i\|_2^2\sim\chi_{\dimension}^2$, the value of $\|X_i\|_2^2$ concentrates around the mean $\dimension$ with deviation of order $O_{\mathbb{P}}(\sqrt{\dimension})$. Therefore, unless $|t-\sqrt{\dimension}|=O(1)$, $n(t)$ is either very close to $0$ or very close to $n$. The only truncation level $t=(1\pm o(1))\sqrt{\dimension}$ that results in a nontrivial subset cannot lead to an efficient tradeoff between covariate magnitude and effective sample size as in (\ref{eq:guess-rate}) of the univariate setting.

\section{Minimax Rates in High Dimension}\label{sec:main}

\subsection{Upper Bound Using Regression Depth}

The failure of truncation according to $\ell_2$ norm does not rule out truncating the data using low-dimensional projections as in the univariate setting. That is, instead of using estimators constructed from $\{(X_i,y_i)\}_{i\in[n]:\|X_i\|_2\geq t}$, we hope to build an estimator using $\{(v^\top X_i,y_i)\}_{i\in[n]: |v^\top X_i|\geq t}$ for all $v\in \mathcal{S}^{\dimension-1}$. While this may not be straightforward by modifying some M-estimator, it turns out the idea goes well with robust estimators based on the regression depth function.

Given a joint probability distribution $\mathbb{P}$ of $(X,y)$ and a regression vector $\beta\in\mathbb{R}^{\dimension}$, the regression depth function \citep{rousseeuw1999regression} is defined by
\begin{equation}
\mathcal{D}(\beta,\mathbb{P}) = \inf_{v\in \mathcal{S}^{\dimension-1}}\mathbb{P}\left(v^\top X(y-X^\top\beta)\geq 0\right). \label{eq:reg-depth}
\end{equation}
The definition (\ref{eq:reg-depth}) is analogous to the well known halfspace depth \citep{tukey1975mathematics} for location parameters. The maximizer of the empirical version of (\ref{eq:reg-depth}) is a robust estimator for the regression vector, and it is known to achieve the error rate $\sqrt{\frac{p}{n}}+\epsilon$ under \Cref{model:huber} \citep{Gao20}. Since the definition (\ref{eq:reg-depth}) is based on all one-dimensional projections of the covariate $X$, a natural modification that uses information of large $v^\top X$ is given by
\begin{equation}
\mathcal{D}(\beta,\mathbb{P},t) = \inf_{v\in \mathcal{S}^{\dimension-1}}\mathbb{P}\left(v^\top X(y-X^\top\beta)\geq 0, |v^\top X|\geq t\right). \label{eq:trunc-depth}
\end{equation}
The empirical version of (\ref{eq:trunc-depth}) is
$$\mathcal{D}(\beta,\mathbb{P}_n,t)=\inf_{v\in \mathcal{S}^{\dimension-1}}\frac{1}{n}\sum_{i=1}^n\indi\{v^\top X_i(y_i-X_i^\top\beta)\geq 0, |v^\top X_i|\geq t\}.$$
For a given direction $v\in \mathcal{S}^{\dimension-1}$, only the subset $\{i\in[n]: |v^\top X_i|>t\}$ is used in the computation of the depth on that direction. On the other hand, $\mathcal{D}(\beta,\mathbb{P}_n,t)$ depends on $\bigcup_{v\in \mathcal{S}^{\dimension-1}}\{i\in[n]: |v^\top X_i|>t\}$, which can be the entire data set when $\dimension$ is large, and thus every data point is informative in the high dimensional setting.
A robust estimator is defined by maximizing the truncated depth function,
\begin{equation}
\wh{\beta} = \argmax_{\widetilde{\beta}  \in \R^\dimension}\mathcal{D}(\widetilde{\beta},\mathbb{P}_n,t). \label{eq:depth-est}
\end{equation}

\begin{Theorem}\label{thm:info-upper}
Consider data generated from \Cref{model:adaptive} and the estimator (\ref{eq:depth-est}) with $t\in[0,\sqrt{0.4\log(n/\dimension)}]$.
For any $\alpha\in(0,1)$, there exist $C,c>0$ such that whenever $\sqrt{\frac{\dimension}{n}}+\epsilon \leq c$, the estimator (\ref{eq:depth-est}) satisfies
\begin{equation}
\|\wh{\beta}-\beta\|_2\leq C\frac{\sigma}{t} \left(\sqrt{\frac{\dimension}{n}}e^{t^2}+\epsilon\right), \label{eq:up-rate}
\end{equation}
with probability at least $1-\alpha$. Thus, by taking $t=\frac{1}{2}\sqrt{\log(n\epsilon^2/\dimension+e)}$, we achieve the error rate (\ref{eq:opt-rate}).
\end{Theorem}

\begin{Remark}
The optimal truncation level $t=\frac{1}{2}\sqrt{\log(n\epsilon^2/\dimension+e)}$ is an increasing function of $\epsilon$. Intuitively, larger covariates are more resilient to the contamination on responses. When $\epsilon$ is unknown, an adaptive estimator can be constructed to achieve the same optimal error rate via a standard Lepski's method \citep{lepskii1991problem,lepskii1992asymptotically,jain2022robust} that selects $t$ from data.
\end{Remark}

\subsection{Lower Bound}\label{sec:lower}

Between the two terms in the optimal rate (\ref{eq:opt-rate}), the first term $\sqrt{\frac{\sigma^2\dimension}{n}}$ can be obtained by a standard lower bound argument \citep{polyanskiy2025information} in a regression model without contamination. On the other hand, deriving the lower bound $\frac{\sigma\epsilon}{\sqrt{\log(n\epsilon^2/\dimension+e)}}$ requires some new technical tool. In the setting of \Cref{model:huber}, the optimal error rate $\sigma\epsilon$ does not depend on the dimension, and the lower bound construction is a simple two-point argument \citep{chen2015robust} based on the fact that for any distributions $P_1$ and $P_2$ satisfying $\TV(P_1,P_2)\leq \frac{\epsilon}{1-\epsilon}$, there exist $Q_1$ and $Q_2$, such that
\begin{equation}
(1-\epsilon)P_1+\epsilon Q_1=(1-\epsilon)P_2+\epsilon Q_2. \label{eq:match-2}
\end{equation}
However, the two-point construction does not lead to the sharp rate $\frac{\sigma\epsilon}{\sqrt{\log(n\epsilon^2/\dimension+e)}}$ in the setting of \Cref{model:adaptive}.

We will derive the lower bound $\frac{\sigma\epsilon}{\sqrt{\log(n\epsilon^2/\dimension+e)}}$ using Fano's inequality \citep{yu1997assouad}. Let $v_1,\cdots,v_m$ be a $\frac{1}{2}$-packing of the unit sphere $\mathcal{S}^{\dimension-1}$. That is, $\|v_j-v_k\|_2>\frac{1}{2}$ for all $j\neq k$. It is known that such a packing exists with $m\geq 2^\dimension$. For each $j\in[m]$, with some $\delta>0$ and some conditional distribution $Q_{j,X}$ to be specified later, we define $\mathbb{P}_j$ to be a joint distribution of $(X,y)$ whose sampling process is given by
\begin{eqnarray}
\label{eq:reg-design-l} X &\sim& \mathcal{N}(0,I_\dimension), \\
\label{eq:reg-huber-l} y\mid X &\sim& (1-\epsilon)\mathcal{N}(\delta X^\top v_j,\sigma^2) + \epsilon Q_{j,X}.
\end{eqnarray}
Then, Fano's inequality gives
\begin{equation}
\inf_{\wh{\beta}}\sup_{\beta,Q}\mathbb{P}\left(\|\wh{\beta}-\beta\|_2\geq\frac{\delta}{4}\right)\geq 1-\frac{n\max_{j,k\in[m]}\KL(\mathbb{P}_j\|\mathbb{P}_k)+\log 2}{\log m}\,, \label{eq:fano}
\end{equation}
where $\KL(P \| Q)$ denotes the KL divergence of $P$ from $Q$.
It suffices to bound $\max_{j,k\in[m]}\KL(\mathbb{P}_j\|\mathbb{P}_k)$ with appropriate choices of the conditional distributions $Q_{1,X},\cdots,Q_{m,X}$. Intuitively, we need to construct these conditional distributions so that $\{(1-\epsilon)\mathcal{N}(\delta X^\top v_j,\sigma^2) + \epsilon Q_{j,X}\}_{j=1}^m$ are close to each other for a typical value of $X$ following (\ref{eq:reg-design-l}). Unlike matching two distributions in (\ref{eq:match-2}), now we need to match $m$ distributions simultaneously.

To this end, let us first consider a simpler problem of matching $m$ distributions without conditioning. Given arbitrary probability distributions $P_1,\cdots,P_m$, our goal is to find $\{Q_j\}_{j=1}^m$ such that $(1-\epsilon)P_j+\epsilon Q_j$ does not depend on $j\in[m]$. Similar to the $m=2$ case, whether matching more than two distributions is possible depends on a quantity that generalize the total variation distance. We define the total variation of $P_1,\cdots,P_m$ as
\begin{equation}
\TV(P_1,\cdots,P_m) = \int \max_{1\leq j\leq m}\frac{dP_j}{d\mu} d\mu -1, \label{eq:tv-m}
\end{equation}
where $\mu$ is a common dominating measure. The definition (\ref{eq:tv-m}) is a special case of the general $f$-divergence of $m$ distributions studied by \cite{duchi2018multiclass}. The following lemma is an extension of the two-point method (\ref{eq:match-2}).

\begin{Lemma}\label{lem:matching}
Suppose the distributions $P_1,\cdots,P_m$ satisfy $\TV(P_1,\cdots,P_m)\leq \frac{\epsilon}{1-\epsilon}$ for some $\epsilon\in[0,1)$. Then, there exist distributions $Q_1,\cdots,Q_m$ such that $(1-\epsilon)P_j+\epsilon Q_j=(1-\epsilon)P_k+\epsilon Q_k$ for all $j,k\in[m]$.
\end{Lemma}

We will apply \Cref{lem:matching} to the conditional distribution $y\mid X\sim \mathcal{N}(X^\top\beta,\sigma^2)$. The following lemma bounds $\TV(P_1,\cdots,P_m)$ when each $P_j$ is a Gaussian location model.

\begin{Lemma}\label{lem:tv-gaussian}
For any $\theta_1,\cdots,\theta_m\in\mathbb{R}$ and any $\sigma >0$, we have
$$\TV\left(\mathcal{N}(\theta_1,\sigma^2),\cdots,\mathcal{N}(\theta_m,\sigma^2)\right)\leq \frac{\max_{1\leq j\leq m}\theta_j-\min_{1\leq j\leq m}\theta_j}{\sqrt{2\pi}\sigma}.$$
\end{Lemma}

\Cref{lem:matching} and \Cref{lem:tv-gaussian} together imply the existence of $\{Q_j\}_{j=1}^m$ such that $(1-\epsilon)\mathcal{N}(\theta_j,\sigma^2)+\epsilon Q_j$ does not depend on $j\in[m]$ as long as all $\{\theta_j/\sigma\}_{j=1}^m$ are close to each other within the order of $\epsilon$. More generally, for arbitrary $\{\theta_j/\sigma\}_{j=1}^m$, we can show that there still exist $\{Q_j\}_{j=1}^m$ such that $\{(1-\epsilon)\mathcal{N}(\theta_j,\sigma^2)+\epsilon Q_j\}_{j=1}^m$ are getting closer by the order of $\epsilon$.

\begin{Corollary}\label{cor:kl-for-fano}
For any $\theta_1,\cdots,\theta_m\in\mathbb{R}$ and any $\sigma>0$, there exist distributions $Q_1,\cdots,Q_m$ such that
\begin{eqnarray*}
&& \KL\left((1-\epsilon)\mathcal{N}(\theta_j,\sigma^2)+\epsilon Q_j\|(1-\epsilon)\mathcal{N}(\theta_k,\sigma^2)+\epsilon Q_k\right) \\
&\leq& 2\left(\frac{2|\theta_j|}{\sigma}-\sqrt{\frac{\pi}{2}}\epsilon\right)_+^2+2\left(\frac{2|\theta_k|}{\sigma}-\sqrt{\frac{\pi}{2}}\epsilon\right)_+^2,
\end{eqnarray*}
for all $j,k\in[m]$.
\end{Corollary}
\begin{proof}
It suffices to consider $\sigma=1$.
Define $\mathcal{J}=\{j\in[m]: |\theta_j|\leq \sqrt{\pi/2}\epsilon\}$. For any $j,k\in\mathcal{J}$, we have $|\theta_j-\theta_k|\leq \sqrt{2\pi}\epsilon$. By \Cref{lem:tv-gaussian}, the total variation of $\{\mathcal{N}(\theta_j,1):j\in\mathcal{J}\}$ is bounded by $\epsilon$. Using \Cref{lem:matching}, we know that there exist $\{Q_j:j\in\mathcal{J}\}$ such that $(1-\epsilon)\mathcal{N}(\theta_j,1)+\epsilon Q_j$ is the same distribution for all $j\in\mathcal{J}$. If $\mathcal{J}\neq \varnothing$, take $\ell$ to be the smallest element in $\mathcal{J}$ and then set $Q_j=Q_{\ell}$ for all $j\notin\mathcal{J}$. Otherwise if $\mathcal{J}= \varnothing$, we set $Q_j=\mathcal{N}(0,1)$ for all $j\notin\mathcal{J}$.

We write $P_j=(1-\epsilon)\mathcal{N}(\theta_j,1)+\epsilon Q_j$ for $j\in[m]$ with the $Q_1,\cdots,Q_m$ constructed above. Suppose $j,k\in\mathcal{J}$, we have $\KL(P_j\|P_k)=0$. For any $j\notin\mathcal{J}$, we must have $|\theta_j|\leq 2|\theta_j|-\sqrt{\pi/2}\epsilon$. Suppose $j\notin \mathcal{J}$ and $k\in\mathcal{J}$, we have
$$\KL(P_j\|P_k)=\KL(P_j\|P_{\ell})\leq \KL(\mathcal{N}(\theta_j,1)\|\mathcal{N}(\theta_{\ell},1))=\frac{1}{2}(\theta_j-\theta_{\ell})^2\leq 2\theta_j^2\leq 2\left(2|\theta_j|-\sqrt{\pi/2}\epsilon\right)_+^2.$$
The same bound holds for $\KL(P_k\|P_j)$ with a similar argument. Suppose $j,k\notin\mathcal{J}$, we have
$$\KL(P_j\|P_k)\leq \KL(\mathcal{N}(\theta_j,1)\|\mathcal{N}(\theta_k,1))\leq \theta_j^2+\theta_k^2\leq \left(2|\theta_j|-\sqrt{\pi/2}\epsilon\right)_+^2+\left(2|\theta_k|-\sqrt{\pi/2}\epsilon\right)_+^2.$$
Thus, the desired bound holds for all the four cases.
\end{proof}

Applying \Cref{cor:kl-for-fano} to $\{\mathcal{N}(\delta X^{\top}v_j,\sigma^2)\}_{j=1}^m$ conditioning on the value of $X$, we know that there exist $Q_{1,X},\cdots,Q_{m,X}$ such that
\begin{eqnarray}
\nonumber && \KL\left((1-\epsilon)\mathcal{N}(\delta X^{\top}v_j,\sigma^2)+\epsilon Q_{j,X}\|(1-\epsilon)\mathcal{N}(\delta X^{\top}v_k,\sigma^2)+\epsilon Q_{k,X}\right) \\
\label{eq:klbd|X} &\leq& 2\left(\frac{2\delta|X^{\top}v_j|}{\sigma}-\sqrt{\frac{\pi}{2}}\epsilon\right)_+^2+2\left(\frac{2\delta|X^{\top}v_k|}{\sigma}-\sqrt{\frac{\pi}{2}}\epsilon\right)_+^2,
\end{eqnarray}
for all $j,k\in[m]$. The bound (\ref{eq:klbd|X}) is zero when both $|X^{\top}v_j|$ and $|X^{\top}v_k|$ are small, which agrees with the intuition behind the upper bound construction (\ref{eq:trunc-depth}) in the sense that the statistical information is from the tail of the one-dimensional projection of the covariate. Recall that $\mathbb{P}_j$ stands for the joint distribution (\ref{eq:reg-design-l}) and (\ref{eq:reg-huber-l}), we can thus bound the Kullback-Leibler divergence on the right hand side of (\ref{eq:fano}) by
\begin{eqnarray}
\nonumber \KL(\mathbb{P}_j\|\mathbb{P}_k) &\leq& 2\mathbb{E}_{X\sim \mathcal{N}(0,I_\dimension)}\left(\frac{2\delta|X^{\top}v_j|}{\sigma}-\sqrt{\frac{\pi}{2}}\epsilon\right)_+^2+2\mathbb{E}_{X\sim \mathcal{N}(0,I_\dimension)}\left(\frac{2\delta|X^{\top}v_k|}{\sigma}-\sqrt{\frac{\pi}{2}}\epsilon\right)_+^2 \\
\nonumber &=& 4\mathbb{E}_{G\sim \mathcal{N}(0,1)}\left(\frac{2\delta|G|}{\sigma}-\sqrt{\frac{\pi}{2}}\epsilon\right)_+^2 \\
\nonumber &=& 8\int_0^{\infty}\left(2\delta t/\sigma-\sqrt{\pi/2}\epsilon\right)_+^2\frac{1}{\sqrt{2\pi}}e^{-t^2/2}dt \\
\nonumber &\leq& 8\sqrt{2}\max_{t>0}\left[\left(2\delta t/\sigma-\sqrt{\pi/2}\epsilon\right)_+^2e^{-t^2/4}\right]\int_0^{\infty}\frac{1}{\sqrt{4\pi}}e^{-t^2/4}dt \\
\label{eq:kl-bd-lb} &=& 4\sqrt{2}\max_{t>0}\left[\left(2\delta t/\sigma-\sqrt{\pi/2}\epsilon\right)_+^2e^{-t^2/4}\right].
\end{eqnarray}
In order that $n\max_{j,k\in[m]}\KL(\mathbb{P}_j\|\mathbb{P}_k)\leq \frac{1}{4}\log m$ so that the right hand side of Fano's inequality (\ref{eq:fano}) is a constant, it suffices to set $\delta$ so that
\begin{equation}
4\sqrt{2}\left(2\delta t/\sigma-\sqrt{\pi/2}\epsilon\right)_+^2e^{-t^2/4}\leq \frac{\dimension\log 2}{4n}, \label{eq:lb-requirement}
\end{equation}
holds for all $t>0$. Rearranging (\ref{eq:lb-requirement}) leads to the choice
\begin{equation}
\delta=\min_{t>0}\left[\frac{\sigma}{2t}\left(\sqrt{\frac{\pi}{2}}\epsilon + \frac{1}{4}\sqrt{\frac{\dimension\log 2}{\sqrt{2}n}}e^{t^2/8}\right)\right]\asymp \sigma\left(\sqrt{\frac{\dimension}{n}}+\frac{\epsilon}{\sqrt{\log(n\epsilon^2/\dimension+e)}}\right), \label{eq:delta-rate-lb-pf}
\end{equation}
which matches the upper bound rate (\ref{eq:opt-rate}). We summarize this lower bound result in the following theorem.

\begin{Theorem}[Information-theoretic Lower Bound]
\label{thm:info-lower}
There exists some universal constant $C>0$ such that
$$\inf_{\wh{\beta}}\sup_{\beta,Q}\mathbb{P}_{\beta,\sigma,Q}\left(\|\wh{\beta}-\beta\|_2\geq C\sigma\left(\sqrt{\frac{\dimension}{n}}+\frac{\epsilon}{\sqrt{\log(n\epsilon^2/\dimension+e)}}\right)\right)\geq\frac{1}{2},$$
where $\mathbb{P}_{\beta,\sigma,Q}$ stands for the data distribution of \Cref{model:adaptive}.
\end{Theorem}

\section{Information-Computation Tradeoffs}\label{sec:SQ}


The rate-optimal estimator (\ref{eq:depth-est}) requires maximizing the truncated depth function, which is a computationally infeasible problem \citep{amenta2000regression} when $\dimension$ is large. On the other hand, the estimator based on median regression,
\begin{equation}
\wh{\beta}_{\mathrm{Median-Regression}}=\argmin_{\widetilde{\beta} \in \R^\dimension}\frac{1}{n}\sum_{i=1}^n|y_i-X_i^\top\widetilde{\beta}|, \label{eq:md-reg-hd}
\end{equation}
can be computed efficiently via linear programming. The statistical error of (\ref{eq:md-reg-hd}) is given by the following theorem.

\begin{Theorem}\label{thm:md-reg-hd}
Consider data generated from \Cref{model:adaptive} and the estimator (\ref{eq:md-reg-hd}).
For any $\alpha\in(0,1)$, there exist $C,c>0$ such that whenever $\sqrt{\frac{\dimension}{n}}+\epsilon \leq c$, the estimator (\ref{eq:md-reg-hd}) satisfies
\begin{equation}
\|\wh{\beta}_{\mathrm{Median-Regression}}-\beta\|_2\leq C\sigma\left(\sqrt{\frac{\dimension}{n}}+\epsilon\right),\label{eq:med-reg-error-rate}
\end{equation}
with probability at least $1-\alpha$.
\end{Theorem}

We note that the above error bound is the same as the minimax rate of estimating $\beta$ under \Cref{model:huber}, but is sub-optimal under \Cref{model:adaptive} by comparing it with the faster rate (\ref{eq:opt-rate}) in terms of the dependence on $\epsilon$.
Furthermore, under the stronger \Cref{model:huber}, 
the asymptotic error of $\wh{\beta}_{\mathrm{Median-Regression}}$ is unbounded even in the univariate case.\footnote{The lower bound follows by considering the following noise contamination for \modelref{model:huber} with $x_0 , \beta_0 \to \infty$: $(1-\epsilon)P_{0,1} + \epsilon D_{x_0, \beta_0 x_0}$, where $D_{x_0,y_0}$ denotes the point mass at $(x_0,y_0)$.}

In this section, we will show evidence that the $O(\sigma\epsilon)$ term in (\ref{eq:med-reg-error-rate}) cannot be improved using a polynomial-time algorithm by establishing a matching statistical query lower bound. The result thus suggests that achieving the optimal rate (\ref{eq:opt-rate}) implies a necessary computational cost.

\subsection{Statistical Query Model}

The statistical query (SQ) model
\citep{kearns1998efficient} is a common framework for
providing rigorous evidence of computational barriers
in high-dimensional statistical problems
\citep{FelGRVX17,DiaKS17,DiaKRS23,DIKR25}. 
The reader is referred to~\cite{Feldman16b} for a survey. 
In the SQ framework, one does not
have direct access to samples generated from some
distribution $P$. Instead, one only has access to
an SQ oracle, which can be interpreted as a
statistic of the form $\frac{1}{n}\sum_{i=1}^n q(X_i)$.
Provided that $q$ is bounded, we have
$$
	\left|\frac{1}{n}\sum_{i=1}^n q(X_i)-\E_{X\sim
		P}[q(X)]\right|
	=
	O_{\mathbb{P}}\Bigl(\frac{1}{\sqrt{n}}\Bigr).
$$
More generally, an SQ oracle responds to a query with some
number that is close to $\E_{X\sim P}[q(X)]$, such
as (but not necessarily) the empirical average over a set
of i.i.d.\ samples. An SQ algorithm is only allowed to
compute an output by adaptively querying the oracle. The
total number of queries made by an algorithm can be
understood as a surrogate for its computational complexity.
This setting naturally includes many optimization
algorithms such as gradient descent and procedures derived
from M-estimators.

We first define the following SQ oracle.

\begin{Definition}[STAT Oracle]
	Let $P$ be a distribution on the domain $\cX$ with sigma-algebra $\cE$.
	A statistical query is a bounded measurable
	function $q : \cX \to [-1, 1]$.
	For a $\tau \in (0,1)$, the $\STAT_{P,\tau}$ oracle 
  responds to the query $q$ with a value $v = \STAT_{P,\tau}(q)$
	such that $|v - \E_{X\sim P} [q(X)]| \leq \tau$.
  We call $\tau$ the tolerance of the SQ oracle.
\end{Definition}

In addition to the STAT oracle above, another popular oracle in the literature is
 the VSTAT oracle~\cite[Definition 2.3]{FelGRVX17}.
Without going into the details, we note that the STAT and VSTAT oracles are quadratically related,
 and hence our super-polynomial lower bounds on the number of queries to the STAT oracle also imply similar lower bounds for the VSTAT oracle.

Note that the definition is abstract and does not
involve actual samples generated by $P$. In a statistical
setting where samples are available, given a query $q$, one
can implement a $\STAT_{P,\tau}$ oracle by reporting the
sample average of $q(X_1),\dots,q(X_n)$ for $n =
	\Theta(1/\tau^2)$ i.i.d.\ samples from the
distribution $P$. Thus, if an SQ algorithm, formally defined below, needs to access
$\STAT_{P,\tau}$ for a small $\tau$, this is interpreted as
requiring higher sample complexity; more broadly, we may
interpret the sample complexity as $n= 1/\tau^c$ for some $c>0$.

More formally, we define $\mathcal{F}$ to be the set of all statistical queries, which is the set of all measurable functions bounded in $[-1,1]$.
An SQ oracle $\STAT_{P,\tau}$ is a map from $\cF \to
	\R$ such that for all $q \in \cF$, it holds that
$|\STAT_{P,\tau}(q) - \E_{X\sim P} [q(X)]|\leq \tau$. We define
$\mathtt{Oracle}_{P,\tau}$ to be the set of all such
oracles, i.e.,
$$
\,\mathtt{Oracle}_{P,\tau} := \left\{ g: \cF \to
	\R \text{ such that } \forall q\in \cF \,,\, |g(q)- \E_P[q]| \leq \tau\right\}.
$$

A statistical query algorithm $\cA:=\cA_{k,\tau}$,
parameterized by the number of queries $k$ and tolerance
$\tau$, interacts with an (unknown) SQ oracle
$\STAT_{P,\tau} \in \,\mathtt{Oracle}_{P,\tau}$ in $k$
rounds as follows. For the $i$-th round with $i \in [k]$,
$\cA$ (randomly) chooses a statistical query $q_i \in \cF$
based on the history $(q_j,v_j)_{j=1}^{i-1}$, and obtains
the value $v_i = \STAT_{P,\tau}(q_i)$. At the end of $k$ rounds, 
the algorithm
outputs a value $\cA(\STAT_{P,\tau})$, where the output
takes values in $\R^\dimension$ for the estimation problem
and $\{0,1\}$ for the testing problem (defined below). We define
$\text{SQ}(k,\tau)$ to be the class of all such SQ
algorithms.

We define $\cD_{\beta,\sigma,\epsilon}$ to be the class of all
distributions on $(X,y)$ that satisfy \Cref{model:adaptive}.

\begin{Definition}[SQ Estimation]
	We say that an SQ algorithm $\cA$ solves linear regression under \Cref{model:adaptive} with 
  error $\delta$, $k$ 
  queries, and tolerance $\tau$, if $\cA \in \text{SQ}(k,\tau)$ and
	\begin{align*}
		\sup_{(\beta,\sigma): \|\beta\|_2 \leq 1, \sigma\in[1/2,1]}\,\,\, \sup_{P \in \cD_{\beta,\sigma,\epsilon}} \,\,\,\sup_{ \STAT_{P,\tau} \in \,\mathtt{Oracle}_{P,\tau} } \bbP\left(\bigl\|\cA (\STAT_{P,\tau}) -\beta \bigr\|_2 >  \sigma\delta\right) \leq 0.1\,,
	\end{align*}
	where the probability is taken over the randomness of the algorithm.
\end{Definition}
That is,  for any $\beta$ and $\sigma$ with $\|\beta\|_2\leq 1$ and $\sigma\in[1/2,1]$\footnote{This constraints make the lower bound only stronger.},
  any $P \in \cD_{\beta,\sigma,\epsilon}$, and
any oracle $\STAT_{P,\tau}$,
 the algorithm outputs $\widehat{\beta}$ such that
   with probability at least $0.9$ over the randomness of the algorithm, $\|\widehat{\beta}- \beta \|_2\leq \sigma\delta$.

\subsection{Statistical Query Lower Bound}
We now present our main computational lower bound.

\begin{Theorem}[SQ Lower Bound]\label{thm:ThmFormalSQ}
	Let the contamination rate $\epsilon \in (0,1/2)$ and 
  let the accuracy threshold $\delta$ be such that $\epsilon/\delta \gtrsim 1$.
   Let the dimension $\dimension\geq 3$ be large enough: 
   $\dimension \gtrsim \left((\epsilon/\delta) \log \dimension\right)^{\Omega(1)}$.
	Any statistical query algorithm $\cA$ that solves linear regression under \Cref{model:adaptive} with 
  error $\delta$, $k$ 
  queries, and tolerance $\tau$ must
	satisfy either
	\begin{itemize}
		\item (large number of queries) $k \geq 2^{\Omega(\dimension^{\Omega(1)})}$ many queries, or
		\item (small tolerance) $\tau \leq O(\dimension^{- \Omega(\epsilon^2/\delta^2)})$.
	\end{itemize}
\end{Theorem}


\Cref{thm:ThmFormalSQ} shows that, in order to
achieve the error bound $\|\wh{\beta}-\beta\|_2\leq
	\sigma\delta$, a ``polynomial-time'' SQ algorithm must use
$\dimension^{\Omega(\epsilon^2/\delta^2)}$ many effective ``samples''. 
In particular, if we restrict attention to algorithms whose sample complexity is polynomial in $\dimension$, this lower bound forces $\epsilon^2/\delta^2 = O(1)$, or equivalently $\delta \gtrsim \epsilon$.
The result can therefore be understood as a computational lower bound showing that a
polynomial-time algorithm cannot achieve an error bound
smaller than a constant multiple of $\sigma\epsilon$, thus providing
strong evidence that the error rate
(\ref{eq:med-reg-error-rate}) achieved by the median
regression estimator (\ref{eq:md-reg-hd}) cannot be improved given
the computational constraint.

The result of \Cref{thm:ThmFormalSQ} is proved by
establishing the hardness of the following hypothesis
testing problem:
\begin{align}
	\label{eq:reg-test-h0} H_0: \qquad&   P = \bbP_0:=\mathcal{N}(0,I_\dimension)\otimes R,                                     \\
	\label{eq:reg-test-h1} H_1: \qquad&   P\in \left\{\mathbb{P}_{\delta v,\sigma,Q^v}:v\in\mathcal{S}^{\dimension-1} \right\},
\end{align}
where $R$ is some distribution on $\mathbb{R}$, and $\mathbb{P}_{\delta v,\sigma,Q^v}$ 
is the distribution specified by \Cref{model:adaptive} with 
$\beta=\delta v$ and $Q=Q^v$ for some $v\in\mathcal{S}^{\dimension-1}$. 
Suppose $R=(1-\epsilon)\cN(0,1)+\epsilon D$ for some distribution $D$. 
Then (\ref{eq:reg-test-h0}) can also be regarded as an instance of \Cref{model:adaptive} with $\beta=0$ and $\sigma = 1$. 
 Therefore, the goal is to test whether $\beta=0$ under 
 (\ref{eq:reg-test-h0}) or $\|\beta\|_2=\delta$ under 
 the alternative (\ref{eq:reg-test-h1}).\footnote{In fact, our computational 
 lower bound holds for the (easier) Bayesian testing problem
  when $v$ is chosen uniformly at random from the unit sphere.}

\begin{Definition}[SQ Testing]
	Let $P_0$ be a distribution and let $\cP$ be a set of distributions over a common domain such that $P_0 \not\in \cP$.
	We say that an SQ algorithm $\cA$ solves the testing problem $H_0: P = P_0$ 
  versus $H_1: P \in \cP$ with $k$ queries and tolerance $\tau$ if
   $\cA \in \text{SQ}(k,\tau)$ and
	\begin{align*}
		\sup_{P \in \{P_0\} \cup \cP} \,\,\,\sup_{ \STAT_{P,\tau} \in \,\mathtt{Oracle}_{P,\tau} } \bbP\left( \cA (\STAT_{P,\tau})  \neq \indi\{P \neq P_0\}\right) \leq 0.1\,,
	\end{align*}
	where the probability is taken over the randomness of the algorithm.
\end{Definition}

Indeed, any SQ algorithm that solves the estimation problem implies an
 SQ test that solves the testing problem (\ref{eq:reg-test-h0})--(\ref{eq:reg-test-h1}).

\begin{Lemma}[Estimation Implies Testing]
\label{lem:est-implies-testing}
Suppose there is an SQ algorithm that solves linear regression under \Cref{model:adaptive} with $k$ 
  queries, tolerance $\tau$, and
  error $\delta/4$.
Then there is an SQ test that solves the testing problem in \eqref{eq:reg-test-h0}-\eqref{eq:reg-test-h1} for the same $\delta$
with $k$ queries and tolerance $\tau$ as long as $\sigma\in[1/2,1]$ and $R = (1 -\epsilon)\cN(0,1) +  \epsilon D$ for some distribution $D$.
\end{Lemma}

Thus, our goal is to construct the parameter $\sigma$ and
the distributions $R$ and $Q$ such that the testing problem
\eqref{eq:reg-test-h0}--\eqref{eq:reg-test-h1} is hard for SQ algorithms.

\subsection{Preliminaries of SQ Framework}

A standard benchmark problem used to establish computational hardness is called Non-Gaussian Component Analysis (NGCA). The goal of NGCA is to test whether a high-dimensional distribution has a one-dimensional direction whose marginal distribution is different from standard Gaussian. We first give a definition of the high-dimensional hidden direction distribution.

\begin{Definition}[High-Dimensional Hidden Direction Distribution] \label{def:high-dim-distribution}
For a unit vector $v \in \mathcal{S}^{\dimension-1}$ and a distribution $A$ on the real line with probability density function $A(\cdot)$, define $P^{A}_v$ to be a distribution over $\R^\dimension$, where $P^{A}_v$ is the product distribution whose orthogonal projection onto the direction of $v$ is $A$, 
and onto the subspace perpendicular to $v$ is the standard $(\dimension{-1})$-dimensional normal distribution. 
That is, $P^{A}_v(X) := A(v^\top X) \phi_{v^{\bot}}(X)$, where $\phi_{v^{\bot}}(X) = \exp\left(-\|X - (v^\top X)v\|_2^2/2\right)/(2\pi)^{(\dimension-1)/2}$.
\end{Definition}

The NGCA refers to the following hypothesis testing problem:
\begin{eqnarray*}
\label{eq:ngca-test0} H_0: && P=\mathcal{N}(0,I_\dimension), \\
\label{eq:ngca-test1} H_1: && P\in\left\{P_v^A: v\in\mathcal{S}^{\dimension-1}\right\}.
\end{eqnarray*}
NGCA was originally proposed in \citep{blanchard2006search}, and its computational hardness for SQ algorithms was established in \cite{DiaKS17}.
In particular,  \cite{DiaKS17} showed that if $A$ matches $m$ moments with $\cN(0,1)$ and $\chi^2(A, \cN(0,1))$ is finite, then all SQ algorithms that solve the testing problem need either $2^{p^{\Theta(1)}}$ many queries or need small tolerance $\tau = O(\dimension^{-\Omega(m)}) \sqrt{\chi^2(A,\cN(0,1))}$.
Due to the generality afforded by the choice of $A$,
NGCA has been used to show computational lower bounds for 
many high-dimensional statistical tasks; see \cite[Chapter 8]{DiaKan22-book}.

To connect our regression problem to NGCA, we observe that the Gaussian linear model (\ref{eq:glm}) can be equivalently defined by the following sampling process,
\begin{equation}
y \sim \cN(0,\sigma^2+\|\beta\|_2^2)\quad\text{and}\quad X\mid y\sim \cN\left(\frac{y}{\sigma^2+\|\beta\|_2^2}\beta, I_p-\frac{1}{\sigma^2+\|\beta\|_2^2}\beta\beta^\top\right).
\end{equation}
In other words, the conditional distribution of $X$ given $y$ is an instance of the high-dimensional hidden direction distribution with direction $v=\beta/\|\beta\|_2$ and non-Gaussian component $A=\cN\left(\frac{\|\beta\|_2}{\sigma^2+\|\beta\|_2^2}y, \frac{\sigma^2}{\sigma^2+\|\beta\|_2^2}\right)$. More generally, with the presence of the contamination component as in \Cref{model:adaptive}, we can still write the conditional distribution of $X$ given $y$ as a high-dimensional hidden direction distribution as long as the contamination distribution $Q_X$ depends on $X$ only through $X^\top\beta$.
However, since we are working with the joint distribution of $(X,y)$ in the regression problem, it is necessary  to extend the NGCA to the following conditional version.

\begin{Definition}[Conditional NGCA]\label{def:ngca-lin-regr}
Let $\mathcal{A}=\{A_y\}_{y\in\mathbb{R}}$ be a family of univariate distributions, and let $R$ be a univariate distribution. Consider the testing problem given access to a distribution $P$ on $\R^\dimension\times \R$:
\begin{eqnarray}
\label{eq:con-ngca-test0} H_0: && P = \mathcal{N}(0,I_\dimension)\otimes R, \\
\label{eq:con-ngca-test1} H_1: && P\in \left\{\mathbb{P}_{v,R}^{\mathcal{A}}: v\in\mathcal{S}^{\dimension-1}\right\},
\end{eqnarray}
where $\mathbb{P}_{v,R}^{\mathcal{A}}$ stands for a distribution of $(X,y)$ with sampling process given by $y\sim R$ and $X\mid y\sim P_v^{A_y}$.
\end{Definition}

Building on \cite{DiaKS17}, the conditional NGCA was first introduced by \cite{DiaKS19} to establish the computational hardness of (outlier)-robust linear regression (in the Huber contamination model).
Like the NGCA, the conditional NGCA is hard when the distributions $\{A_y\}_{y\in\R}$ match $m$ moments with standard Gaussian for all $y\in\R$, in which case any SQ algorithm solving\footnote{The notion of success of an SQ test is defined \emph{mutatis mutandis} as in \eqref{eq:reg-test-h0}-\eqref{eq:reg-test-h1} by substituting the appropriate $H_0$ and $H_1$.} conditional NGCA either needs (roughly) $2^{\dimension^{\Omega(1)}}$ many queries or at least a single query with tolerance at most (roughly) $\dimension^{-\Omega(m)}$. We state this result as the following lemma; for details, see \cite[Lemma 5.7]{DiaKPPS21}. 

\begin{Lemma}[SQ hardness of Conditional NGCA under matching moments~\cite{DiaKS19,}]
\label{lem:sq-hardness-conditional-ngca}
Consider the testing problem in Definition \ref{def:ngca-lin-regr} and let $m \in \N$.
Suppose that for every $y \in \R$, the distribution $A_y$ matches $m$ moments with $\cN(0,1)$.
Then, every SQ algorithm that solves the testing problem with $k$ queries and tolerance $\tau$ must satisfy either
\begin{itemize}
    \item (large number of queries)  $k \geq \frac{2^{\Omega(\dimension^{\Omega(1)})}}{\dimension^{(m+1)/4}}$, or
    \item (small tolerance)  $\tau \le O(1)\frac{ \sqrt{\E_{y \sim R}[\chi^2(A_y,\cN(0,1))]}}{\dimension^{(m+1)/8}}  $.
\end{itemize}
\end{Lemma}


With \Cref{lem:sq-hardness-conditional-ngca}, it suffices to construct $\sigma,R,Q$ so that the testing problem (\ref{eq:reg-test-h0})-(\ref{eq:reg-test-h1}) is an instance of the conditional NGCA given in \Cref{def:ngca-lin-regr}. To this end, we will need Hermite polynomials as part of the technical tools.

\paragraph{Gaussian Hermite Analysis}

For a $k \in \N$, we use $\he_k: \R \to \R$ to denote the $k$-th normalized probabilist's polynomial, which is a degree-$k$ polynomial with definition
$$\he_k(x) := \frac{1}{\sqrt{k!}} (-1)^k e^{x^2/2} \frac{d^k}{dx^k} e^{-x^2/2}.$$
We shall use the following facts about Hermite polynomials.
\begin{Fact}[See, for example, \cite{Szego89,Odo14}]
    \label{fact:Hermite}
    Let $G \sim \cN(0,1)$. 
    Then Hermite polynomials satisfy the following:
    \begin{enumerate}
        \item Hermite polynomials $\{\he_0,\he_1,\dots,\he_k\}$ form a basis of polynomials of degree up to $k$.
        \item For all $i,j \in \{0\} \cup \N$: $\E[\he_i(G)\he_j(G)] = \indi\{i=j\}$.
        \item For any $i  \in \N$, $\rho \in (0,1)$ and $\mu \in \R$, we have that $\E[\he_i( \rho \mu + \sqrt{1 - \rho^2} G)] = \rho^i \he_i(\mu)$.
        \item There exists a constant $C>0$ such that for all $x\in \R$, $i \in \N$, $|\he_i(x)| \leq (C(1 + |x|))^i$.
    \end{enumerate}
\end{Fact}

\subsection{Construction of a Conditional NGCA Instance}\label{sec:con-NGCA}

In this section, we show that the testing problem (\ref{eq:reg-test-h0})-(\ref{eq:reg-test-h1}) is an instance of conditional NGCA (\Cref{def:ngca-lin-regr}, where each $A_y$ matches $m = \Theta(\epsilon^2/\delta^2)$ many moments with $\cN(0,1)$.
We take $\sigma^2=1-\delta^2$ and $Q^v_X=E_{v^\top X}$ in (\ref{eq:reg-test-h1}) for some distribution $E_{v^\top X}$ that only depends on $X$ through the one-dimensional projection $v^\top X$. 
These choices are motivated by the fact that 
the joint density function of $(X,y)$ given some $v$ under (\ref{eq:reg-test-h1}) is
\begin{equation}
\mathrm{pdf}(X,y)=\phi_{v^{\bot}}(X)\phi(x')\left((1-\epsilon)\phi(y;\delta x',1-\delta^2)+\epsilon E_{x'}(y)\right), \label{eq:joint-den-ngca}
\end{equation}
where $x'=v^\top X$ and (with slight abuse of notation) we write $E_{x'}(\cdot)$ for the density function of the distribution $E_{x'}$. The joint distribution (\ref{eq:joint-den-ngca}) implies that the marginal of $y$ is given by the density function
\begin{equation}
R(y) = (1-\epsilon)\int\phi(x')\phi(y;\delta x',1-\delta^2)dx'+\epsilon \int\phi(x')E_{x'}(y)dx'. \label{eq:ngca-R}
\end{equation}
Note that this (\ref{eq:ngca-R}) satisfies the condition of \Cref{lem:est-implies-testing}
since $\int\phi(x')\phi(y;\delta x',1-\delta^2)dx'=\phi(y)$, and thus we can use (\ref{eq:ngca-R}) as the distribution $R$ in  (\ref{eq:reg-test-h0}).
Moreover, the factorization of (\ref{eq:joint-den-ngca}) implies that the distribution of $X\mid y$, in the subspace orthogonal to $v$ is the standard multivariate $(\dimension-1)$-dimensional Gaussian, while along the $v$ direction, the conditional distribution of $x'=v^\top X$ given $y$ has density function
\begin{equation}
A_y(x') = \frac{(1-\epsilon)\phi(x')\phi(y;\delta x',1-\delta^2)+\epsilon \phi(x')E_{x'}(y)}{R(y)}. \label{eq:ngca-Ayx'}
\end{equation}
In other words, we have $X\mid y\sim P_v^{A_y}$ under (\ref{eq:joint-den-ngca}), and thus the testing problem (\ref{eq:reg-test-h0})-(\ref{eq:reg-test-h1}) is an instance of conditional NGCA with $\{A_y\}_{y\in\mathbb{R}}$ and $R$ given by (\ref{eq:ngca-Ayx'}) and (\ref{eq:ngca-R}). By \Cref{lem:sq-hardness-conditional-ngca}, it suffices to construct the conditional distribution $E_{x'}$ so that the induced $A_y$ in (\ref{eq:ngca-Ayx'}) matches moments of $\mathcal{N}(0,1)$ and its chi-squared divergence from $\mathcal{N}(0,1)$ is also bounded.

From now on, we will drop the prime symbol on $x'$ and just write $x$ for notational simplicity whenever the context is clear.

\paragraph{An Alternative Factorization}

The component $\phi(x)E_x(y)$ in the numerator of (\ref{eq:ngca-Ayx'}) stands for the joint distribution of $(x,y)$ (recall that $x$ stands for $v^{\top}X$) when the data is drawn from the contamination distribution. Instead of directly constructing $E_x$, we write
\begin{equation}
\phi(x)E_x(y) = D(y)D_y(x), \label{eq:diff-factor-contam}
\end{equation}
and we will construct $D(\cdot)$, the marginal density of $y$, and $D_y(\cdot)$, the conditional density of $x$ given $y$. We will show there exists a construction such that
\begin{enumerate}[label=(\Roman*), ref=\Roman*]
    \item \label[cond]{item:intro-marginals}
    The marginal of $x$ under $D(y)D_y(x)$ is $\cN(0,1)$.
    \item \label[cond]{item:intro-moments-Ay} For each $y \in \R$, the induced $A_y$ matches $m = \Theta(\epsilon^2/\delta^2)$ moments with $\cN(0,1)$.
\end{enumerate}
To satisfy the first condition, we need
$$\int D(y)D_y(x)dy =\phi(x)\text{ for all }x\in\mathbb{R}.$$
The simplest choice to make here is to take
\begin{equation}
D_y(x) = \phi(x) + f_y(x), \label{eq:D=phi+f}
\end{equation}
with the $f_y:\R \to \R$ being some tiny fluctuations, which are small enough so that the resulting (conditional) distribution is valid (i.e., $|f_y(x)|\leq \phi(x)$ and $\int f_y(x)dx = 0)$, but flexible enough to satisfy the second criterion \Cref{item:intro-moments-Ay}.
Under this choice, the criterion \Cref{item:intro-marginals} is equivalent to the following conditions on $\{f_y\}_{y \in \R}$:
\begin{enumerate}[label=(I.\alph*), ref=I.\alph*]
    \item \label[cond]{item:f-mean-zero-intro} $\int f_y(x)dx = 0$ for all $y \in \R$.
    \item \label[cond]{item:f-D-zero-intro} $\int D(y) f_y(x)dy = 0$ for all $y \in \R$.
    \item \label[cond]{item:f-bound} $|f_y(x)|\leq \phi(x)$ for all $x,y \in \R$.
\end{enumerate}
The first condition \Cref{item:f-mean-zero-intro} demands that the average of $f_y$ should be zero.
A natural choice would be to take $f_y$ to be a linear combination of mean-zero functions of $x$.
This suggests we should take $f_y(x) = \sum_{i \in I} b_i(y) g_i(x)$, where $\int g_i(x)dx = 0$ for all $i\in I$.
The second condition \Cref{item:f-D-zero-intro} requires that $\int f_y(x)D(y) dy = 0$.
This suggests that $f_y(x)D(y)$, as a function of $y$, should be a linear combination of mean-zero functions of $y$. 
Applying this suggestion to the aforementioned choice of $f_y(x) D(y) = \sum_{i \in I} b_i(y) D(y) g_i(x)$,
 we should take $b_i(y) D(y)$ to be a mean-zero function of $y$, for example, $\phi(y)$ multiplied by a polynomial $a_i(y)$ that has mean zero under $\cN(0,1)$.
With these two choices, the candidate fluctuation has the following form:
\begin{align}
\label{item:f-candidate-form-intro}
f_y(x) = \sum_{i \in I} \frac{\phi(y)}{D(y)} a_i(y) g_i(x)\,,
\end{align}
where for all $i \in I$: $\int a_i(y) \phi(y)dy = 0$ and $\int g_i(x) dx = 0$.
The choice of the polynomials $a_i$ and functions $g_i$ would come from the second criterion \Cref{item:intro-moments-Ay} above.

\paragraph{Moment Matching Condition}

The moment condition (\ref{item:intro-moments-Ay}) imposes that $A_y$ matches $m$ moments with $\cN(0,1)$. By \Cref{fact:Hermite}, this is equivalent to
\begin{equation}
\int \he_j(x)A_y(x)dx=\int \he_j(x)\phi(x)dx=0, \label{eq:matching-hermite}
\end{equation}
for all $j\in[m]$ and $y\in\mathbb{R}$. To check (\ref{eq:matching-hermite}), we note that the formula (\ref{eq:ngca-Ayx'}) can also be written as
\begin{equation}
A_y(x)=\frac{(1-\epsilon)\phi(y)\phi(x;\delta y,1-\delta^2)+\epsilon D(y)D_y(x)}{(1-\epsilon)\phi(y)+\epsilon D(y)}, \label{eq:Ay-alt-form}
\end{equation}
where we have used the identity $\phi(x)\phi(y;\delta x,1-\delta^2)=\phi(y)\phi(x;\delta y,1-\delta^2)$ and (\ref{eq:diff-factor-contam}). With $D_y(x) = \phi(x) + f_y(x)$, the condition (\ref{eq:matching-hermite}) is reduced to
\begin{equation}
\int f_y(x)\he_j(x)dx=-\frac{(1-\epsilon)\phi(y)}{\epsilon D(y)}\int \he_j(x)\phi(x;\delta y,1-\delta^2)dx=-\frac{(1-\epsilon)\phi(y)\delta^j\he_j(y)}{\epsilon D(y)},\label{eq:matching-fy-hermite}
\end{equation}
where the last identity follows by \Cref{fact:Hermite}. On the other hand, the integral $\int f_y(x)\he_j(x)dx$ under the candidate form of $f_y(x)$ in (\ref{item:f-candidate-form-intro}) is
$$\sum_{i \in I} \frac{\phi(y)}{D(y)} a_i(y) \int\he_j(x)g_i(x)dx.$$
For this to be equal to (\ref{eq:matching-fy-hermite}) for all $j\in[m]$ and $y\in\mathbb{R}$, we set $I$, $a_i(y)$ and $g_i(y)$ according to $I=[m]$, $a_i(y)=-\frac{(1-\epsilon)\delta^i\he_i(y)}{\epsilon}$ and $\int\he_j(x)g_i(x)dx=\indi\{i=j\}$. Observe that $a_i(y)$ does have zero expectation under the standard Gaussian measure (as required above). Thus, our final choice of $f_y$ has the following form:
\begin{equation}
f_y(x) = \frac{\phi(y)}{D(y)}\sum_{i=1}^ma_i(y)g_i(x), \label{eq:fy-final-form}
\end{equation}
where $g_i$ satisfies $\int\he_j(x)g_i(x)dx=\indi\{i=j\}$ for all $j\in[m]$ and $\int g_i(x)dx = 0$.  Observe that such $g_i$'s imply \Cref{item:f-mean-zero-intro,,item:f-D-zero-intro,,item:intro-moments-Ay}.
Restating these conditions on $g_i$'s more compactly and accounting for the remaining constraint of \Cref{item:f-bound} in \Cref{item:g-bound-intro} (with justification to follow shortly), the criteria \Cref{item:intro-moments-Ay,item:intro-marginals} are 
satisfied if there exist functions $\{g_i\}_{i=1}^m$ such that
\begin{enumerate}[label=(G.\Roman*),ref=G.\Roman*]
    \item \label[cond]{item:g-hermite-moments-intro}For all $i \in [m]$ and $j \in [0:m]$, $\int \he_j(x)g_i(x)dx = \indi\{i=j\}$.
    \item \label[cond]{item:g-bound-intro} It holds that $\sum_{i=1}^m \sup_x \frac{|g_i(x)|}{\phi(x)} \frac{\delta^i}{\epsilon} \leq 1$.
\end{enumerate}
The second condition here, \Cref{item:g-bound-intro}, is imposed to satisfy the condition \Cref{item:f-bound} that $|f_y(x)| \leq \phi(x)$. Indeed, since the condition \Cref{item:f-bound} is implied by
\begin{equation}
\sum_{i=1}^m\phi(y)|a_i(y)|\sup_x\frac{|g_i(x)|}{\phi(x)}\leq D(y)\text{ for all }y\in\mathbb{R},\label{eq:D-pointwise-lb}
\end{equation}
we could take $D(y)$ that satisfies (\ref{eq:D-pointwise-lb}). Such a $D(y)$ can be a valid density function as  
long as $\kappa:= \sum_{i=1}^m \sup_x \frac{|g_i(x)|}{\phi(x)} \cdot \int \phi(y)|a_i(y)|dy \leq 1$, and observing that $\int \phi(y)|a_i(y)|dy = \E[|a_i(G)|] = \frac{(1-\epsilon)}{\epsilon}\delta^i \E[|\he_i(G)|]\leq 
\frac{\delta^i}{\epsilon} \sqrt{\E[\he_i(G)^2]} \leq \frac{\delta^i}{\epsilon}$ gives us \Cref{item:g-bound-intro}.
To be precise, we take $D(y)$ to be 
\begin{align}
\label{eq:qprime-def}
D(y) = ( 1 -  \kappa) \phi(y) + \sum_{i=1}^m\phi(y)|a_i(y)|\sup_x\frac{|g_i(x)|}{\phi(x)} \,.
\end{align}

\paragraph{Existence of Appropriate $g_i$'s} Observe that achieving either \Cref{item:g-hermite-moments-intro} or \Cref{item:g-bound-intro} in isolation is rather easy; 
for example, \Cref{item:g-hermite-moments-intro} is satisfied by $g_i(x) = \he_i(x) \phi(x)$, but it does not satisfy \Cref{item:g-bound-intro} because $g_i(x)/\phi(x) = \he_i(x)$ is unbounded. 
We now show that both conditions can be achieved simultaneously.
\begin{Proposition}[Existence of $g_i$'s using LP Duality]
\label{prop:existence-g-intro}
There exists a positive constant $C$ such that, 
for every $m \in \N$,
there exist measurable functions $g_1,\dots,g_m$ satisfying
\begin{itemize}
     \item For all $i \in [m]$ and $j \in [0:m]$, $\int \he_j(x)g_i(x)dx = \indi\{i=j\}$.

 \item For all $i \in [m]$, $\sup_x \frac{|g_i(x)|}{\phi(x)}\leq (Cm)^{i/2}$.
  \end{itemize} 
\end{Proposition}

First, observe that this proposition implies both  \Cref{item:g-hermite-moments-intro} and \Cref{item:g-bound-intro}: the first is trivial, and the second follows from the following simple calculation:
$\sum_{i=1}^m (Cm)^{i/2} \delta^i/\epsilon \leq \sum_{i=1}^{m}  (Cm \delta^2/\epsilon^2)^{i/2} \leq 1$, where the last inequality holds as long as $m \leq \frac{\epsilon^2}{4C \delta^2}$, and we will then use $m \leq \frac{\epsilon^2}{4C \delta^2}$ as a condition for $m$.
We provide the formal proof of \Cref{prop:existence-g-intro} in \Cref{sec:pf-SQ}, but present a proof sketch below:

\begin{proof}[Proof Sketch of \Cref{prop:existence-g-intro}]
We write $B_i=(Cm)^{i/2}$ and define $\mathcal{A}_i$ to be the set of all functions $r$ on $\mathbb{R}$ such that $\sup_x|r(x)|\leq B_i$.
By writing $g_i(x)=\phi(x)r_i(x)$, it suffices to show the existence of $r_i$ in $\mathcal{A}_i$ such that $\int \he_j(x)r_i(x)\phi(x)dx = \indi\{i=j\}$ for all $j \in [0:m]$. Consider the linear operator $T$ that maps each $r\in\mathcal{A}_i$ to a $(m+1)$-dimensional vector whose entries are given by $\{\int \he_{j}(x)r(x)\phi(x)dx\}_{j\in[0:m]}$. We then define $\mathcal{B}_i$ to be the set of all such projections, i.e., $B_i=\{T(r):r\in\mathcal{A}_i\}$. Our goal is to show that $e_{i}\in\mathbb{R}^{m+1}$, the vector with all entries zero expect the $i$-th coordinate being one (the coordinates are indexed by $[0:m]$), belongs to the set $\mathcal{B}_i$. First, it can be shown that $\cB_i$ is a compact convex set.
Hence, the hyperplane separation theorem implies that $e_{i}$ does belong to $\cB_i$ unless there is a vector $u \in \R^{m+1}$ such that $\min_{w \in \cB_i} u^\top w > u^\top e_{i}$.
Hence, it suffices to show that for all $u \in \R^{m+1}$, there exists a $w \in \cB_i$ (or an $r \in \cA_i$) such that  $ u^\top w = u^\top T(r) \leq u^\top e_{i}$.

For each $u=(u_0,u_1,\cdots,u_m)^\top\in\mathbb{R}^{m+1}$, 
there corresponds a polynomial $f(x)=\sum_{j=0}^mu_j\he_j(x)$ whose degree is at most $m$. It can be checked that $u^\top T(r)=\mathbb{E}[f(G)r(G)]$ 
and $u^\top e_{i}=\mathbb{E}[f(G)\he_i(G)]$ with $G\sim \cN(0,1)$. 
Thus, it suffices to show that for all degree-$m$ polynomial $f$, there exists an $r\in\mathcal{A}_i$ such that $\mathbb{E}[f(G)r(G)]\leq\mathbb{E}[f(G)\he_i(G)]$. 
To this end, we take $r=-B_i\text{sign}(f)$ (which does belong to $\cA_i$), 
and the condition becomes $-B_i\mathbb{E}[|f(G)|]\leq\mathbb{E}[f(G)\he_i(G)]$. 
Equivalently,
$$B_i \geq  \sup_{\text{deg}(f)\leq m; \,\,f \neq 0 }\frac{\mathbb{E}[f(G)\he_i(G)]}{\mathbb{E}[|f(G)|]}.$$
Using the hypercontractivity of Gaussian polynomials, we show in \Cref{lem:sup-poly} (stated in \Cref{sec:pf-SQ}) that the supremum above is bounded by $2\sup_{|x|\leq \sqrt{32m}}|\he_i(x)|$, which is at most $B_i=(Cm)^{i/2}$.
\end{proof}

\paragraph{Chi-Squared Bound}

With the construction of $A_y$ outlined above, we still need to bound its expected chi-squared divergence from $\cN(0,1)$ in order to apply the result of \Cref{lem:sq-hardness-conditional-ngca}. This is done by the following lemma.
\begin{Lemma}\label{lem:chi-squared-bd}
Consider $R=(1-\epsilon)\phi+\epsilon D$ and the conditional distribution $A_y$ in (\ref{eq:Ay-alt-form}) constructed with (\ref{eq:qprime-def}), (\ref{eq:D=phi+f}) and (\ref{eq:fy-final-form}) using functions $\{g_i\}_{i\in[m]}$ satisfying the two properties in \Cref{prop:existence-g-intro}. As long as $m \leq \frac{\epsilon^2}{4C \delta^2}$, there exists another constant $C'>0$ such that $\mathbb{E}_{y\sim R}\left[\chi^2(A_y,\cN(0,1))\right]\leq C'm$.
\end{Lemma}

Plugging the chi-squared bound to \Cref{lem:sq-hardness-conditional-ngca} with $m=\floor{\frac{\epsilon^2}{4C \delta^2}}$, we can conclude that an SQ algorithm that solves the testing problem (\ref{eq:con-ngca-test0})-(\ref{eq:con-ngca-test1}), which is equivalent to (\ref{eq:reg-test-h0})-(\ref{eq:reg-test-h1}), either needs $2^{\dimension^{\Omega(1)}}$ many queries or at least a single query with tolerance at most  $p^{-\Omega(\epsilon^2/\delta^2)}$. This then leads to the conclusion of \Cref{thm:ThmFormalSQ} by \Cref{lem:est-implies-testing}. The details of the proof will be given in \Cref{sec:pf-SQ}.

\subsection{Lower Bounds Against Low-Degree Polynomial Tests}

In this section, we show that the computational lower bounds in \Cref{thm:ThmFormalSQ} for the class of SQ algorithms also apply to low-degree polynomial tests.
We refer the reader to \cite{Hopkins-thesis,KunWB19,BreBHLS21,Wein25-survey} for further details and present a brief background below.

We consider a Bayesian version of the testing problem described in \eqref{eq:reg-test-h0}-\eqref{eq:reg-test-h1} by imposing a prior distribution $H$ on $v$, where $H$ is supported on the unit sphere $\cS^{p-1}$.
Formally, the samples 
$(X_i,y_i)_{i=1}^n$ are generated  i.i.d.\  from a distribution $P$,
with the following two disjoint hypotheses:
\begin{align}
\label{eq:low-degree-testing0}
    H_0: \qquad&P = \cN(0,I_\dimension) \otimes R\\ 
    H_1: \qquad&P = \mathbb{P}_{\delta v,\sigma,Q^v}, \text{where $v \sim H$}\,.
\label{eq:low-degree-testing}
\end{align}
Here, $R$ and $\mathbb{P}_{\delta v,\sigma,Q^v}$ are defined as in \eqref{eq:reg-test-h0}-\eqref{eq:reg-test-h1}.
We choose $H$ to be the uniform distribution over a large set of nearly orthogonal unit vectors.\footnote{A qualitatively similar result holds when $H$ is the uniform distribution on the sphere.}

We restrict our attention to tests that are polynomials in the input samples $(X_i,y_i)_{i=1}^n$.
A degree-$k$ $n$-sample polynomial test for this problem is a degree-$k$ polynomial $h: \R^{(\dimension+1) \times n} \to \R$ and a threshold $t \in \R$. The corresponding test evaluates the polynomial $h$ on the samples and returns $H_0$ if and only if $h((X_1,y_1),\dots,(X_n,y_n)) > t$.
In this context, the degree of the polynomial serves  as a proxy for the runtime:
roughly speaking, the class of degree-$k$ polynomials is interpreted as a proxy for the class of all algorithms running in time $(n\dimension)^{\widetilde{\Theta}_{n\dimension}(k)}$.

A standard way for analyzing the performance of (polynomial) tests is to show that the variance of the test under both null and the alternate is much smaller than the difference in the expected values under the null and the alternate; once this is established, Chebyshev's inequality implies that both the Type-I and Type-II errors are bounded. 
The following definition is a necessary condition for a valid test whose performance is analyzed in this way.
\footnote{It is not sufficient because the variance under the alternative is not considered.} 
\begin{Definition}[Low-degree good-distinguisher polynomial]
  Consider the testing problem in \eqref{eq:low-degree-testing0}-\eqref{eq:low-degree-testing} and denote the null distribution by $\bbP_0 := \cN(0,I_\dimension) \otimes R$.
  We say that a degree-$k$ $n$-sample polynomial $h$ is a $\tau$-distinguisher
  if 
    \begin{align}
      \left|  \E_{ Z \sim \bbP_0^{\otimes n}}[h(Z)] - \E_{v \sim H}\left[ \E_{Z \sim \bbP_{\delta v, \sigma, Q^v}^{\otimes n}}[h(Z)]\right]\right| \geq \tau \sqrt{\var_{Z \sim \bbP_0^{\otimes n}}[h(Z)]}\,.
    \end{align}
    The \emph{advantage} of the polynomial $h$ is defined to be the largest $\tau$ that satisfies the definition above.
\end{Definition}

The advantage $\tau$ of a test corresponds to the signal to noise ratio of the test:
if the advantage of the polynomial is less than $c$ for a constant $c>0$,
then vanishing Type-I and Type-II error cannot be achieved by the standard analysis mentioned above.
Thus, upper bounding the advantage of a test by a constant, say $1$, is viewed as its failure of the test for the distinguishing problem.

Our main result in this section 
provides evidence that polynomial-time algorithms must use $\dimension^{\Omega(\epsilon^2/\delta^2)}$ samples as opposed to the information-theoretic sample complexity of $\frac{\dimension}{\epsilon^2} e^{\epsilon^2/\delta^2}$ samples.
Recall that in the framework of low-degree polynomial tests, polynomial-time algorithms correspond to degree-$k$ tests with small $k$:   the community standard allows $k$ to be at most $\mathrm{poly}(\log(n))$.
In fact, the result rules out sub-exponential
time algorithms that use less than $\dimension^{\Omega(\epsilon^2/ \delta^2)}$ samples: it shows that $k$ must be larger than $\dimension^{\Omega(1)}$, which corresponds to the time complexity of $(n\dimension)^{\widetilde{\Theta}(\dimension^{\Omega(1)})}$.
This result is obtained by using the relationship between low-degree polynomial tests and statistical query algorithms established in \cite{BreBHLS21}.
\begin{Corollary}
Consider the testing problem in \eqref{eq:low-degree-testing0}-\eqref{eq:low-degree-testing} with $\sigma,R,Q$ constructed in \Cref{sec:con-NGCA}.
Let $\dimension \gtrsim (\epsilon^2/\delta^2)^{\Omega(1)}$ and $k < \dimension^{\Omega(1)}$ and $\epsilon/\delta\gtrsim 1$. Then, there exists some distribution $H$ on $\cS^{p-1}$ such that
if a polynomial is an $n$-sample degree-$k$ $\tau$-distinguisher with $\tau \geq 1$ (that is, the polynomial has advantage at least $1$),
we must have  $ n \gtrsim \dimension^{\Omega(\frac{\epsilon^2}{\delta^2})}$.
\end{Corollary}
\begin{proof}
This follows from \Cref{thm:ThmFormalSQ}, where we establish that underlying $\{A_y\}_{y \in \R}$ match $m = \Omega(\epsilon^2/\delta^2)$ moments, and  \cite[Corollary 6.4]{DiaKPPS21} by setting $c = 1/4$ therein.
\end{proof}

\section{Discussion}\label{sec:disc}

\subsection{Comparison with Different Contamination Models}\label{sec:models}
In this section, we contrast \modelref{model:adaptive} and \modelref{model:huber} with related contamination models that have been studied in the literature.
By saying that one contamination model is ``weaker'' (resp.\ ``stronger'') than another, we mean that it defines a subset (resp.\ superset) of distributions.

\paragraph{Weaker Contamination Models.} We begin with contamination models that are special cases of \modelref{model:adaptive}.
The ``weakest'' type of contamination we consider has clean covariates and a response contamination mechanism that is oblivious to the covariates.
There are two natural ways to define such an oblivious contamination of the responses.
The first contaminates only the additive noise, independently of the features.

\begin{Model}[Oblivious Contamination in Responses (I)]
\label{model:obv-1}
The pairs $\{(X_i,y_i)\}_{i=1}^n$ are independently drawn according to
\begin{align*}
X_i &\sim \cN(0,I_\dimension),\\
y_i\mid X_i &\sim  (1 - \epsilon) \cN(X_i^\top \beta,\sigma^2)  + \epsilon (\Dirac{X_i^\top \beta} \circledast Q),
\end{align*}
where $Q$ is an arbitrary distribution over $\R$, $\Dirac{z}$ denotes a point mass at $z$, and $\circledast$ denotes convolution between distributions.
Equivalently, $y_i = X_i^\top \beta + (1-V_i)z_i + V_i\gamma_i$,
where $X_i \sim \cN(0,I_\dimension)$, $V_i \sim \mathrm{Bernoulli}(\epsilon)$, $z_i \sim \cN(0,\sigma^2)$, and $\gamma_i \sim Q$, all mutually independent.
\end{Model}

This model has been studied extensively; see, for example, \cite{Tsakonas14, JaiTK14,BhatiaJK15,BhatiaJKK17, 
SugBRJ19,pesme2020online, d2021consistent,Steurer21Outliers,DiaGKLP25-oblivious}.
In particular, under \Cref{model:obv-1}, the minimax rate of estimating $\beta$ under $\ell_2$ norm is $\sigma\sqrt{\frac{\dimension}{n(1-\epsilon)^2}}$,
at least when $n  = \Omega\left(\frac{\dimension}{(1- \epsilon)^2}\right)$~\cite{Steurer21Outliers}.

In the second oblivious contamination model, the response $y$ is directly replaced by an arbitrary value (again oblivious of $X$) with probability $\epsilon$, as opposed to contaminating the additive noise.

\begin{Model}[Oblivious Contamination in Responses (II)]
\label{model:obv-2}
The pairs $\{(X_i,y_i)\}_{i=1}^n$ are independently drawn according to
\begin{align*}
X_i &\sim \cN(0,I_\dimension),\\ 
y_i \mid X_i &\sim (1 - \epsilon) \cN(X_i^\top \beta, \sigma^2) + \epsilon Q,
\end{align*}
where $Q$ is an arbitrary distribution over $\R$.
Equivalently, $y_i = (1 - V_i)(X_i^\top \beta + z_i) + V_i\gamma_i$,
where $X_i \sim \cN(0,I_\dimension)$, $V_i \sim \mathrm{Bernoulli}(\epsilon)$, $z_i \sim \cN(0,\sigma^2)$, and $\gamma_i \sim Q$, all mutually independent.
\end{Model}

While we are not aware of prior work that studies this specific contamination model,
it suffers from the same drawback as \modelref{model:obv-1}: the contamination cannot depend on the covariates, unlike in \modelref{model:adaptive}.
As we show in \Cref{prop:different-models}, both of the above contamination models are special cases of \modelref{model:adaptive}.

\looseness=-1\paragraph{Stronger Contamination models.} Having considered contamination models that are weaker than \modelref{model:adaptive}, we now turn to stronger contamination models.
In the first strengthening,
the covariates are still clean as in \modelref{model:adaptive}, but the contamination rate may depend on $X$ and need not be uniformly bounded by $\epsilon$; we  require only that the \emph{average} contamination rate is at most $\epsilon$.

\begin{Model}[Non-Uniform Contamination in Responses]
\label{model:non-uni}
The pairs $\{(X_i,y_i)\}_{i=1}^n$ are independently drawn according to
\begin{align*}
X_i &\sim \cN(0,I_\dimension),\\
y_i \mid X_i &\sim (1 - \epsilon_{X_i}) \cN(X_i^\top \beta, \sigma^2) + \epsilon_{X_i} Q_{X_i},
\end{align*}
where $Q_{X_i}$ is an arbitrary conditional distribution depending on $X_i$ and $\E[\epsilon_{X_i}] = \epsilon$.
\end{Model}

Observe that this model is incomparable to the classical Huber contamination model (\Cref{model:huber}): the density of the clean distribution might fall below $(1 - \epsilon) P_{\beta,\sigma}$; on the other hand, the covariates are always clean.
In matrix notation, we observe $(X,y)$ with 
$$y = X \beta + Z + \Gamma,$$
where $X$ is an $n\times \dimension$ matrix with independent $\cN(0,1)$ entries, $Z \sim \cN(0, \sigma^2 I_n)$ independent of $X$, and $\Gamma$ is a random vector with independent coordinates whose expected number of nonzero coordinates is at most $\epsilon n$ (and which may depend on the $i$-th row of $X$ but not on $Z$).
A slightly stronger adversarial model is studied in \cite{dalalyan2019outlier}, where the contamination vector $\Gamma$ is allowed to be a completely arbitrary $\epsilon n$ sparse vector (and in particular may depend arbitrarily on both $X$ and $Z$).
In this setting, \cite{dalalyan2019outlier} obtain an error rate of order $\sigma\left(\sqrt{\frac{\dimension}{n}}+\epsilon\right)$ for this problem and argue its optimality by appealing to the lower bound of \cite{Gao20}.
However, \cite{Gao20} proved the lower bound under Huber contamination, which as argued earlier is incomparable to the setting in \cite{dalalyan2019outlier}.
Nevertheless, we show in \Cref{prop:different-models} that similar ideas do yield a matching lower bound under \Cref{model:non-uni} and hence the setting in \cite{dalalyan2019outlier}.

We now consider a different strengthening of \modelref{model:adaptive}.
Observe that \modelref{model:adaptive} can be viewed as a special case of \modelref{model:huber} in which the marginal distribution of the covariates under $Q$ is clean, i.e.\ exactly $\cN(0,I_\dimension)$.
The next model allows the marginal distribution of $X$ under $Q$ to be corrupted as well, provided it remains absolutely continuous with respect to the Gaussian design and its density is uniformly bounded.

\begin{Model}[Huber Contamination with Bounded Marginal Likelihood]
\label{model:hub-bdd-marg}
The pairs $\{(X_i,y_i)\}_{i=1}^n$ are independently drawn according to
\begin{align*}
(X_i,y_i) &\sim (1-\epsilon)P_{\beta,\sigma} + \epsilon Q\\
\text{such that for all } x \in \R^\dimension:&\quad \epsilon q_X(x) \leq \phi(x),
\end{align*}
where $P_{\beta,\sigma}$ is the Gaussian linear model (\ref{eq:glm}), $Q$ is an arbitrary distribution over $(X,y)$, $q_X$ is the marginal density of $X$ under $Q$, and $\phi$ denotes the density of $\cN(0,I_\dimension)$.
\end{Model}

One can generalize this model by requiring $\epsilon q_X(x) \leq B \phi(x)$ for some finite $B \ge 1$, but for simplicity we restrict attention to the case $B=1$.

While this model might seem artificial, we mention it because of its relevance to \cite{chinot2020erm}.
Building on the aforementioned work of   \cite{dalalyan2019outlier} on linear regression, \cite{chinot2020erm} studies a general convex ERM problem in which the marginals are assumed to be clean and the responses are corrupted by an $\epsilon n$-sparse vector.
For the particular task of linear regression, \cite[Theorem~7]{chinot2020erm} is used there to claim a minimax lower bound of order $\sigma\left(\sqrt{\frac{\dimension}{n}}+\epsilon\right)$, implying inconsistency for any constant $\epsilon$.
However, the contamination model considered in the statement of  \cite[Theorem~7]{chinot2020erm} is in fact
\Cref{model:adaptive},
and hence the statement of \cite[Theorem~7]{chinot2020erm} is incorrect in light of our \Cref{thm:info-upper}.
The error in the proof of \cite[Theorem~7]{chinot2020erm} is that the contaminated distribution used there alters the marginal distribution of $X$; 
In fact, their construction is an instance of \modelref{model:hub-bdd-marg}.
Consequently, the lower bound argument in \cite{chinot2020erm} does not directly apply to the clean marginal setting, which is the setting of interest for us as well as \cite{chinot2020erm}.
Fortunately, as shown below in \Cref{prop:different-models}, the broader claim of \cite{chinot2020erm} continues to hold for linear regression by embedding an instance of \modelref{model:non-uni}, which does have clean marginals.

Finally, we consider the well-known TV contamination model to relate it with \modelref{model:non-uni}.

\begin{Model}[TV Contamination]
\label{model:TV}
The pairs $\{(X_i,y_i)\}_{i=1}^n$ are independently drawn according to
\begin{align*}
(X_i,y_i) &\sim Q, \qquad \text{ such that } \TV(P_{\beta,\sigma}, Q) \leq \epsilon,
\end{align*}
where $Q$ is an arbitrary joint distribution over $\R^\dimension \times \R$ and $P_{\beta,\sigma}$ is the Gaussian linear model (\ref{eq:glm}).
\end{Model}

\Cref{fig:models} relates these different contamination models to each other and highlights which of them permit consistent estimation.

\begin{Proposition}
\label{prop:different-models}
The following statements hold:
\begin{enumerate}[label=(\roman*)]
  \item \modelref{model:obv-1} and  \modelref{model:obv-2} are weaker than \modelref{model:adaptive}.
  \item \modelref{model:non-uni} is weaker than \modelref{model:TV} and stronger than \modelref{model:adaptive}.
  \item \modelref{model:hub-bdd-marg} is weaker than \modelref{model:huber} and stronger than \modelref{model:adaptive}.
  \item The minimax rate of estimation under \modelref{model:non-uni} and \modelref{model:hub-bdd-marg} is $\sigma\left(\sqrt{\frac{\dimension}{n}}+\epsilon\right)$.
\end{enumerate}
\end{Proposition}

\begin{figure}[htbp]
\centering
\usetikzlibrary{arrows.meta,shapes.geometric} 
\tikzset{
  model/.style={
    draw,
    ellipse,
    minimum width=1.4cm,
    minimum height=0.9cm,
    align=center
  },
  arrowsF/.style={
    -{Stealth[length=3mm]},
    thick,
    smooth,
    draw=black
  },
  consistent/.style=
  { fill=green!10!white
  },
  inconsistent/.style={
    fill=red!10!white
  }
}
\begin{tikzpicture}[scale=1]
\def\layerA{-3}
\def\layerB{0.5}
\def\layerC{4}
\def\layerD{8.5}
\def\rowA{2}
\def\rowB{1}
\def\rowC{-0.5}
\def\rowD{-1.5}
\def\rowE{-2.5}
  \node[model,dashed,consistent] (Ob1) at (\layerA,  \rowA) {Oblivious I \\
  {\small (\Cref{model:obv-1})}};
  \node[model,dashed,consistent] (Ob2) at (\layerA,  \rowC) {Oblivious II\\
  {\small (\Cref{model:obv-2})}};
  \node[model,dashed,consistent] (Adap) at (\layerB,  \rowB) {Adaptive\\{\small (\Cref{model:adaptive})}\\
  Our focus};
  \node[model, inconsistent] (HubBdd) at ( \layerC, \rowC) {Huber with\\Bounded Marginal\\ Likelihood 
  {\small (\Cref{model:hub-bdd-marg})}};
  \node[model,dashed, inconsistent] (NonUni) at (\layerC,  \rowA) {Non-Uniform\\ {\small (\Cref{model:non-uni})}};
  \node[model, inconsistent] (Hub) at (\layerD,  \rowC) {Huber\\{\small (\Cref{model:huber})}};
  \node[model, inconsistent] (TV) at ( \layerD,  \rowA) {TV\\{\small (\Cref{model:TV})}};
  \draw[arrowsF]    (Ob1) to (Adap);
  \draw[arrowsF]    (Ob2) to (Adap);
  \draw[arrowsF]    (Adap) to (HubBdd);
  \draw[arrowsF]    (Adap) to (NonUni);
  \draw[arrowsF]    (HubBdd) to (Hub);
  \draw[arrowsF]    (Hub) to (TV);
  \draw[arrowsF]    (NonUni) to (TV);
\end{tikzpicture}
\caption{Comparison of different contamination models. An arrow from Model A to Model B indicates that the former is a weaker contamination model.
A green shade indicates that the model permits consistency for any fixed $\epsilon$ that is sufficiently small (say $\epsilon<1/5$), while the red shade indicates that consistency is not possible for any fixed $\epsilon$.
A dashed oval indicates that the covariates are clean, i.e., $X \sim \cN(0,I_p)$.
See \Cref{prop:different-models} for a precise statement.
}
\label{fig:models}
\end{figure}
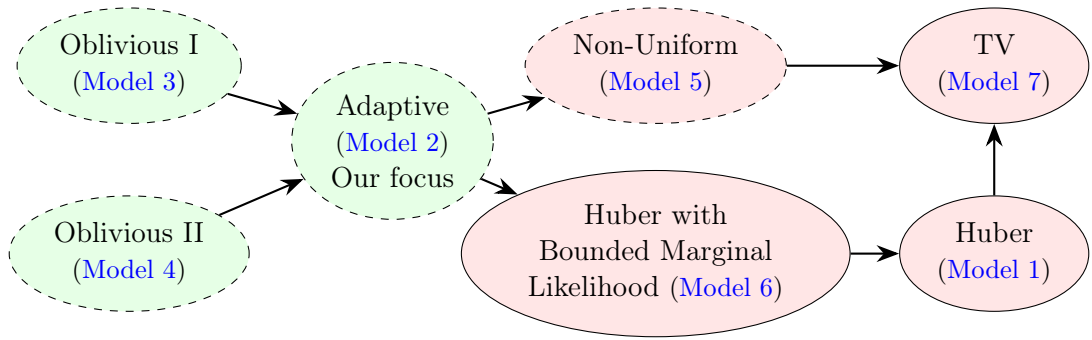
To summarize, among the seven variations, \modelref{model:adaptive} is the strongest contamination setting that permits consistent estimation for a constant level of contamination proportion.



\subsection{Effects of the Covariate Distribution}

\Cref{model:adaptive} assumes that covariates are drawn from $\cN(0,I_\dimension)$, which turns out to have a critical effect on the minimax rate (\ref{eq:opt-rate}). In this section, we show that a different covariate distribution may result in a different minimax rate. We will consider the class of generalized Gaussian distribution as a specific example, though the same analysis can be extended to a more general class of distributions.

\begin{Definition}[Generalized Gaussian Distribution] Given some $\gamma>0$. A $\dimension$-dimensional random vector follows a spherical generalized Gaussian distribution, $X\sim \mathcal{GN}_{\gamma}(0,I_\dimension)$, if the density function of the one-dimensional projection $v^\top X$ is given by
\begin{equation}
\frac{d}{dt}\mathbb{P}\left(v^\top X\leq t\right) \propto \exp\left(-\frac{|t|^{\gamma}}{2}\right),
\end{equation}
for all $v\in\mathcal{S}^{\dimension-1}$.
\end{Definition}

The generalized Gaussian distribution covers the standard Gaussian with $\gamma=2$ and the spherical Lapalce distribution with $\gamma=1$. The dependence of the minimax rate on $\gamma$ is given by the theorem below.

\begin{Theorem}\label{thm:tail}
Consider data generated from \Cref{model:adaptive} with $\mathcal{N}(0,I_\dimension)$ replaced by $\mathcal{GN}_{\gamma}(0,I_\dimension)$. For any $\alpha\in(0,1)$, there exist $C,c>0$ such that whenever $\sqrt{\frac{\dimension}{n}}+\epsilon\leq c$, the estimator (\ref{eq:depth-est}) with $t=\left(\log(n\epsilon^2/\dimension+e)/3\right)^{1/\gamma}$ satisfies
\begin{equation}
\|\wh{\beta}-\beta\|_2\leq C\sigma\left(\sqrt{\frac{\dimension}{n}}+\frac{\epsilon}{\left(\log(n\epsilon^2/\dimension+e)\right)^{1/\gamma}}\right),
\end{equation}
with probability at least $1-\alpha$. Moreover, there exists some $C'>0$ such that
$$\inf_{\wh{\beta}}\sup_{\beta,Q}\mathbb{P}_{\beta,\sigma,Q}\left(\|\wh{\beta}-\beta\|_2\geq C'\sigma\left(\sqrt{\frac{\dimension}{n}}+\frac{\epsilon}{\left(\log(n\epsilon^2/\dimension+e)\right)^{1/\gamma}}\right)\right)\geq\frac{1}{2},$$
where $\mathbb{P}_{\beta,\sigma,Q}$ stands for the data distribution of \Cref{model:adaptive} with $\mathcal{N}(0,I_\dimension)$ replaced by $\mathcal{GN}_{\gamma}(0,I_\dimension)$.
\end{Theorem}

The dependence on $\gamma$ indicates that a covariate distribution with a heavier tail leads to a faster minimax rate. This agrees with our intuition in the one-dimensional setting that a pair $(X_i,y_i)$ with a larger $|X_i|$ has more information under contamination. In the high-dimensional setting, the subset $\{i\in[n]:|v^\top X_i|>t\}$ is expected to have a larger cardinality for each $v\in\mathcal{S}^{\dimension-1}$ when the tail is heavier. This phenomenon is unique to robust estimation. In comparison, when $\epsilon=0$ and the data does not have contamination, the minimax rate of estimating $\beta$ does not depend on $\gamma$.

It is also interesting to note that the covariate distribution does not affect the minimax rate of \Cref{model:huber}. For the class of generalized Gaussian distributions, the minimax rate is always $\sigma\left(\sqrt{\frac{\dimension}{n}}+\epsilon\right)$ under \Cref{model:huber}, regardless of the value of $\gamma$. The additional information resulted from the tail of the covariates is only available when the covariates do not have contamination.

\section{Proofs}\label{sec:pf}

\subsection{Proofs of Upper Bound Results}

This section presents proofs of \Cref{thm:1d}, \Cref{thm:info-upper} and \Cref{thm:md-reg-hd}.

\subsubsection{Proof of \Cref{thm:1d}}

We write $z_i=y_i-\beta X_i$. Then, the derivative of the objective function is
$$G_n(\wt{\beta})=\frac{1}{n}\sum_{i=1}^nX_i\text{sign}((\wt{\beta}-\beta)X_i-z_i)\indi\{|X_i|\geq t\}.$$
To show $|\wh{\beta}-\beta|\leq r$, it suffices to show that $G_n(\wt{\beta})<0$ for all $\wt{\beta}\leq  \beta-r$ and $G_n(\wt{\beta})>0$ for all $\wt{\beta}\geq \beta+r$. Since $G_n(\wt{\beta})$ is monotone, we only need to show $G_n(\beta-r)<0$ and $G_n(\beta+r)>0$. Chebyshev's inequality implies that $G_n(\beta+r)\geq \mathbb{E}G_n(\beta+r) - \frac{C}{\sqrt{n}}$ with high probability, where
$$\mathbb{E}G_n(\beta+r)\geq (1-\epsilon)\mathbb{E}X\text{sign}(rX-\sigma Z)\indi\{|X|\geq t\}-\epsilon\mathbb{E}|X|\indi\{|X|\geq t\},$$
with $X,Z\overset{iid}\sim \mathcal{N}(0,1)$. Define $f(r)=\mathbb{E}X\text{sign}(rX-\sigma Z)\indi\{|X|\geq t\}$, and it is clear that $f(0)=0$ and
$f'(r)=\frac{2}{\sigma}\mathbb{E}X^2\phi\left(\frac{rX}{\sigma}\right)\indi\{|X|\geq t\}$, where $\phi(\cdot)$ is the standard Gaussian density. When $r\leq\frac{\sigma}{2t}$, we have
$$f'(r)\geq \frac{2}{\sigma}\mathbb{E}X^2\phi\left(\frac{rX}{\sigma}\right)\indi\{t\leq |X|\leq  2t\}\geq \frac{2\phi(1)}{\sigma}\mathbb{E}X^2\indi\{t\leq |X|\leq  2t\},$$
which implies
$$f(r)\geq \left(r\wedge\frac{\sigma}{2t}\right)\frac{2\phi(1)}{\sigma}\mathbb{E}X^2\indi\{t\leq |X|\leq  2t\}.$$
Combining the above bounds, we know that $G_n(\beta+r)>0$ whenever
\begin{equation}
r\wedge \frac{\sigma}{2t} \geq \frac{\sigma}{2\phi(1)}\left(\epsilon \frac{\mathbb{E}|X|\indi\{|X|\geq t\}}{\mathbb{E}X^2\indi\{t\leq |X|\leq  2t\}} + \frac{C}{\sqrt{n}\mathbb{E}X^2\indi\{t\leq |X|\leq  2t\}}\right). \label{eq:thm1dcrude}
\end{equation}
We will bound the two terms on the right hand side of (\ref{eq:thm1dcrude}). When $t>10$, $\frac{\mathbb{E}|X|\indi\{|X|\geq t\}}{\mathbb{E}X^2\indi\{t\leq |X|\leq  2t\}}\leq t^{-1}\left(1-\frac{\mathbb{E}|X|\indi\{|X|> 2t\}}{\mathbb{E}|X|\indi\{|X|\geq  t\}}\right)^{-1}\leq C_1t^{-1}$. When $t\leq 10$, $\frac{\mathbb{E}|X|\indi\{|X|\geq t\}}{\mathbb{E}X^2\indi\{t\leq |X|\leq  2t\}}$ is a constant and thus we still have $\frac{\mathbb{E}|X|\indi\{|X|\geq t\}}{\mathbb{E}X^2\indi\{t\leq |X|\leq  2t\}}\leq C_2t^{-1}$. We also have $\mathbb{E}X^2\indi\{t\leq |X|\leq  2t\}\geq t^2\mathbb{P}(t\leq |X|\leq  2t)\geq C_3te^{-t^2/2}$ for $t\geq 1$ and $\mathbb{E}X^2\indi\{t\leq |X|\leq  2t\}\geq C_4t\geq C_4te^{-t^2/2}$ for $t<1$. Therefore, a sufficient condition for (\ref{eq:thm1dcrude}) is
$r\wedge \frac{\sigma}{2t}\geq C_5\frac{\sigma}{t} \left(\epsilon+ \frac{1}{\sqrt{n}e^{-t^2/2}}\right)$, which is implied by taking $r=C_5\frac{\sigma}{t} \left(\epsilon+ \frac{1}{\sqrt{n}e^{-t^2/2}}\right)$ under the condition that $\epsilon+ \frac{1}{\sqrt{n}e^{-t^2/2}}$ is sufficiently small. With the same choice of $r$, it can be shown that $G_n(\beta-r)<0$ holds with high probability by the same argument. This completes the proof.

\subsubsection{Proof of \Cref{thm:info-upper}}

For any $v\in \mathcal{S}^{\dimension-1}$, 
we write $\mathcal{D}_v(\widetilde{\beta},\mathbb{P},t) = \mathbb{P}\left(v^\top X(y-X^\top\widetilde{\beta})\geq 0, |v^\top X|\geq t\right)$, so that $\mathcal{D}(\widetilde{\beta},\mathbb{P},t) = \inf_{v\in \mathcal{S}^{\dimension-1}}\mathcal{D}_v(\widetilde{\beta},\mathbb{P},t)$. 
We also write $(X,y)\sim P_{\beta}$ for the sampling process $X\sim \mathcal{N}(0,I_\dimension)$ and $y\mid X\sim \mathcal{N}(X^\top\beta,\sigma^2)$ by dropping the dependence on $\sigma$ for notational simplicity. Note that the marginal distributions of $X$ under $\mathbb{P}$ and $P_{\beta}$ are the same, but $y \mid X$ follows (\ref{eq:reg-huber}) under $\mathbb{P}$. Then, 
we have the following bound for an arbitrary $\widetilde{\beta}$:
\begin{eqnarray*}
\nonumber && |\mathcal{D}(\wt{\beta},\mathbb{P},t)-\mathcal{D}(\wt{\beta},P_{\beta},t)| \\
\nonumber &\leq& \sup_{v\in \mathcal{S}^{\dimension-1}}\left|\mathcal{D}_v(\wt{\beta},\mathbb{P},t)-\mathcal{D}_v(\wt{\beta},P_{\beta},t)\right| \\
\nonumber &=& \sup_{v\in \mathcal{S}^{\dimension-1}}\left|\mathbb{E}\left[\left(\mathbb{P}\left(v^\top X(y-X^\top\wt{\beta})\geq 0 \mid X\right)-P_{\beta}\left(v^\top X(y-X^\top\wt{\beta})\geq 0 \mid X\right)\right)\indi\{|v^\top X|\geq t\}\right]\right| \\
\nonumber &\leq& \epsilon \sup_{v\in \mathcal{S}^{\dimension-1}}\mathbb{P}\left(|v^\top X|\geq t\right) \\
&=& 2\epsilon(1-\Phi(t)),
\end{eqnarray*}
where $\Phi(\cdot)$ stands for the CDF of $\mathcal{N}(0,1)$. Maximizing over $\wt{\beta}$ gives
\begin{equation}
\sup_{\wt{\beta} \in \R^\dimension}|\mathcal{D}(\wt{\beta},\mathbb{P},t)-\mathcal{D}(\wt{\beta},P_{\beta},t)| \leq 2\epsilon(1-\Phi(t)). \label{eq:perturb-bd}
\end{equation}
A standard VC-dimension calculation (see Section 6.1 of \cite{Gao20}) gives
\begin{equation}
\sup_{\wt{\beta} \in \R^\dimension}|\mathcal{D}(\wt{\beta},\mathbb{P},t)-\mathcal{D}(\wt{\beta},\mathbb{P}_n,t)| \leq C\sqrt{\frac{\dimension}{n}}, \label{eq:vc-bd}
\end{equation}
with high probability. Moreover, we claim that
\begin{equation}
\sup_{\wt{\beta} \in \R^\dimension}\mathcal{D}(\wt{\beta},P_{\beta},t)=\mathcal{D}(\beta,P_{\beta},t)=1-\Phi(t). \label{eq:max-depth}
\end{equation}
To see why (\ref{eq:max-depth}) is true, we first note that $\sup_{\wt{\beta} \in \R^\dimension}\mathcal{D}(\wt{\beta},P_{\beta},t)\geq \mathcal{D}(\beta,P_{\beta},t)=1-\Phi(t)$ is straightforward. Additionally, we also have $\sup_{\wt{\beta} \in \R^\dimension}\mathcal{D}(\wt{\beta},P_{\beta},t)\leq \sup_{\wt{\beta} \in \R^\dimension}\left(\frac{1}{2}\mathcal{D}_v(\wt{\beta},P_{\beta},t)+\frac{1}{2}\mathcal{D}_{-v}(\wt{\beta},P_{\beta},t)\right)=1-\Phi(t)$, which leads to (\ref{eq:max-depth}). Combining (\ref{eq:perturb-bd}) and (\ref{eq:vc-bd}),  and using the definition of $\wh{\beta}$, we have
\begin{eqnarray*}
\mathcal{D}(\wh{\beta},P_{\beta},t) &\geq& \mathcal{D}(\wh{\beta},\mathbb{P},t) - 2\epsilon(1-\Phi(t)) \\
&\geq& \mathcal{D}(\wh{\beta},\mathbb{P}_n,t) - 2\epsilon(1-\Phi(t)) - C\sqrt{\frac{\dimension}{n}} \\
&\geq&  \mathcal{D}(\beta,\mathbb{P}_n,t) - 2\epsilon(1-\Phi(t)) - C\sqrt{\frac{\dimension}{n}} \\
&\geq& \mathcal{D}(\beta,P_{\beta},t) - 4\epsilon(1-\Phi(t)) - 2C\sqrt{\frac{\dimension}{n}}.
\end{eqnarray*}
Together with (\ref{eq:max-depth}), we have
\begin{equation}
1-\Phi(t) - \mathcal{D}(\wh{\beta},P_{\beta},t) \leq 4\epsilon(1-\Phi(t)) + 2C\sqrt{\frac{\dimension}{n}}. \label{eq:before-curv}
\end{equation}
For any $\wt{\beta}$, we have
\begin{eqnarray}
\nonumber && 1-\Phi(t) - \mathcal{D}(\wt{\beta},P_{\beta},t) \\
\nonumber &=& \sup_{v\in \mathcal{S}^{\dimension-1}}\left(1-\Phi(t)-\mathcal{D}_v(\wt{\beta},P_{\beta},t)\right) \\
\nonumber &=& \sup_{v\in \mathcal{S}^{\dimension-1}}\left(1-\Phi(t)-\mathbb{E}\left[\Phi\left(\frac{v^\top XX^\top(\beta-\wt{\beta})}{\sigma|v^\top X|}\right)\indi\{|v^\top X|\geq t\}\right]\right) \\
\label{eq:take-u} &\geq& \mathbb{E}\left(\Phi(\|\wt{\beta}-\beta\|_2|Z|/\sigma)\indi\{|Z|>t\}\right)-(1-\Phi(t)) \\
\nonumber &=& g(\|\wt{\beta}-\beta\|_2/\sigma)-g(0),
\end{eqnarray}
where $g(\delta)=\mathbb{E}\left(\Phi(\delta|Z|)\indi\{|Z|\geq t\}\right)$ with $Z\sim \mathcal{N}(0,1)$, and the inequality (\ref{eq:take-u}) is by taking $v=\frac{\wt{\beta}-\beta}{\|\wt{\beta}-\beta\|_2}$.
Therefore, (\ref{eq:before-curv}) further implies
\begin{equation}
g(\|\wh{\beta}-\beta\|_2/\sigma) - g(0) \leq 4\epsilon(1-\Phi(t)) + 2C\sqrt{\frac{\dimension}{n}}. \label{eq:before-curv2}
\end{equation}
To lower bound $g(\delta)-g(0)$, we first compute the derivative
$$g'(\delta)=\mathbb{E}\left(|Z|\phi(\delta|Z|)\indi\{|Z|\geq t\}\right)\geq \mathbb{E}\left(|Z|\phi(\delta|Z|)\indi\{2t>|Z|\geq t\}\right)\geq t\phi(1)\mathbb{P}\left(2t>|Z|\geq t\right),$$
whenever $\delta\leq (2t)^{-1}$. By the monotonicity of $g(\delta)$, we have
$$g(\delta)-g(0)\geq \left(\delta\wedge\frac{1}{2t}\right)t\phi(1)\mathbb{P}\left(2t>|Z|\geq t\right).$$
Together with (\ref{eq:before-curv2}), we have
\begin{equation}
\frac{\|\wh{\beta}-\beta\|_2}{\sigma}\wedge \frac{1}{2t}\leq \frac{4\epsilon(1-\Phi(t)) + 2C\sqrt{\frac{\dimension}{n}}}{t\phi(1)\mathbb{P}\left(2t>|Z|\geq t\right)} \leq \frac{C_1}{t}\left(\epsilon+\sqrt{\frac{\dimension}{n}}e^{t^2}\right), \label{eq:usedlatertail}
\end{equation}
where we have used $\frac{(1-\Phi(t))}{\mathbb{P}\left(2t>|Z|\geq t\right)}\leq C_2$ and $\mathbb{P}\left(2t>|Z|\geq t\right)\geq C_3e^{-t^2}$ for all $t>0$. We therefore obtain the desired bound when $\epsilon+\sqrt{\frac{\dimension}{n}}e^{t^2}$ is sufficiently small.

\subsubsection{Proof of \Cref{thm:md-reg-hd}}

To prove \Cref{thm:md-reg-hd}, we need the following lemma whose proof will be given in the end of the section.
\begin{Lemma}\label{lem:con-in-medreg}
Consider independent $X_1,\cdots,X_n\sim \mathcal{N}(0,I_\dimension)$ and $z_1,\cdots,z_n\sim \mathcal{N}(0,\sigma^2)$. There exists some constant $C>0$, such that
\begin{eqnarray*}
\sup_{u\in \mathcal{S}^{\dimension-1}}\frac{1}{n}\sum_{i=1}^n\left(\mathbb{E}|X_i^\top u|-|X_i^\top u|\right) &\leq& C\sqrt{\frac{\dimension}{n}}, \\
\sup_{u\in \mathcal{S}^{\dimension-1}}\frac{1}{n}\sum_{i=1}^n\left(\mathbb{E}\left(|z_i-rX_i^\top u|-|z_i|\right)-|z_i-rX_i^\top u|+|z_i|\right) &\leq& Cr\sqrt{\frac{\dimension}{n}}.
\end{eqnarray*}
hold with high probability for any $\sigma,r>0$.
\end{Lemma}

\begin{proof}[Proof of \Cref{thm:md-reg-hd}]
We write the objective function as $L_n(\beta)=\frac{1}{n}\sum_{i=1}^n|y_i-X_i^\top\beta|$. In order to show that $\|\wh{\beta}-\beta\|_2\leq r$, it suffices to show $\inf_{\wt{\beta}: \|\wt{\beta}-\beta\|_2\geq r}\left(L_n(\wt{\beta})-L_n(\beta)\right)>0$. By convexity, we only need to show $\inf_{u\in \mathcal{S}^{\dimension-1}}\left(L_n(\beta+ru)-L_n(\beta)\right)>0$. We note that the data generating process of \Cref{model:adaptive} can be described as first sampling $X_i\sim \mathcal{N}(0,I_\dimension)$ and $\gamma_i\sim\text{Bernoulli}(\epsilon)$, and then sample the response according to $y_i|(X_i,\gamma_i=0)\sim \mathcal{N}(X_i^\top\beta,\sigma^2)$ or $y_i|(X_i,\gamma_i=1)\sim Q_{X_i}$. Writing $z_i=y_i-X_i^\top\beta$, we have $z_i|(X_i,\gamma_i=0)\sim \mathcal{N}(0,\sigma^2)$. We also define $\mathcal{I}=\{i\in[n]:\gamma_i=0\}$ and $\mathcal{O}=[n]\backslash\mathcal{I}$. Then, for any $u\in \mathcal{S}^{\dimension-1}$, we have
\begin{eqnarray}
\nonumber L_n(\beta+ru)-L_n(\beta) &=& \frac{1}{n}\sum_{i=1}^n\left(|z_i-rX_i^\top u|-|z_i|\right) \\
\nonumber &=& \frac{1}{n}\sum_{i\in\mathcal{I}}\left(|z_i-rX_i^\top u|-|z_i|\right) + \frac{1}{n}\sum_{i\in\mathcal{O}}\left(|z_i-rX_i^\top u|-|z_i|\right) \\
\label{eq:pf-reg-med-1} &\geq& \frac{1}{n}\sum_{i\in\mathcal{I}}\mathbb{E}\left(|z_i-rX_i^\top u|-|z_i|\right) \\
\label{eq:pf-reg-med-2} && - \sup_{u\in \mathcal{S}^{\dimension-1}}\frac{1}{n}\sum_{i\in\mathcal{I}}\left(\mathbb{E}\left(|z_i-rX_i^\top u|-|z_i|\right)-|z_i-rX_i^\top u|+|z_i|\right) \\
\label{eq:pf-reg-med-3} && - r\frac{1}{n}\sum_{i\in\mathcal{O}}\mathbb{E}|X_i^\top u| - r\sup_{u\in \mathcal{S}^{\dimension-1}}\frac{1}{n}\sum_{i\in\mathcal{O}}\left(\mathbb{E}|X_i^\top u|-|X_i^\top u|\right).
\end{eqnarray}
We can write (\ref{eq:pf-reg-med-1}) as $\frac{|\mathcal{I}|}{n}f(r)$ with $f(r)=\mathbb{E}\left(|z_i-rX_i^\top u|-|z_i|\right)$. Directly calculation gives $f(0)=f'(0)=0$ and $f''(r)=2\sigma^{-1}\mathbb{E}Z^2\phi(rZ/\sigma)\geq 2\sigma^{-1}\phi(1)\mathbb{E}Z^2\indi\{|Z|\leq 1\}$ for $r\leq \sigma$ with $Z\sim \mathcal{N}(0,1)$. Therefore, $f(r)\geq C_1\frac{r^2\wedge \sigma^2}{\sigma}$ with $C_1=\phi(1)\mathbb{E}Z^2\indi\{|Z|\leq 1\}$ being a constant. Applying \Cref{lem:con-in-medreg}, we can bound the magnitudes of (\ref{eq:pf-reg-med-2}) and (\ref{eq:pf-reg-med-3}) by $C_2\frac{r}{n}\sqrt{d|\mathcal{I}|}$ and $C_3\left(\frac{r|\mathcal{O}|}{n}+\frac{r}{n}\sqrt{d|\mathcal{O}|}\right)$ with high probability, respectively. Therefore, $\inf_{u\in \mathcal{S}^{\dimension-1}}\left(L_n(\beta+ru)-L_n(\beta)\right)>0$ holds whenever
\begin{equation}
C_1\frac{|\mathcal{I}|}{n}\frac{r^2\wedge \sigma^2}{\sigma} > C_2\frac{r}{n}\sqrt{d|\mathcal{I}|} + C_3\left(\frac{r|\mathcal{O}|}{n}+\frac{r}{n}\sqrt{d|\mathcal{O}|}\right). \label{eq:pf-reg-med-4}
\end{equation}
Since $|\mathcal{O}|\sim\text{Binomial}(n,\epsilon)$, we have $|\mathcal{O}|\leq C_4 n\epsilon$ with high probability. Thus, (\ref{eq:pf-reg-med-4}) is implied by $r^2\wedge \sigma^2> C_5r\sigma\left(\sqrt{\frac{\dimension}{n}}+\epsilon\right)$, which holds by taking $r=C_5\sigma\left(\sqrt{\frac{\dimension}{n}}+\epsilon\right)$ under the condition that $\sqrt{\frac{\dimension}{n}}+\epsilon\leq c$.
\end{proof}

To prove \Cref{lem:con-in-medreg}, we need the following result, which is Theorem 4.12 of \cite{ledoux2013probability}.
\begin{Lemma}\label{lem:contraction}
Let $F:\mathbb{R}_+\rightarrow\mathbb{R}_+$ be convex and increasing. Let $\psi_i:\mathbb{R}\rightarrow\mathbb{R}$, $i=1,\cdots,n$, be contractions such that $\psi_i(0)=0$. Then, for any bounded subset $T\subset\mathbb{R}^n$,
$$\mathbb{E}F\left(\frac{1}{2}\sup_{t\in T}\left|\sum_{i=1}^n\delta_i\psi_i(t_i)\right|\right)\leq \mathbb{E}F\left(\sup_{t\in T}\left|\sum_{i=1}^n\delta_i t_i\right|\right),$$
where $\delta_1,\cdots,\delta_n$ are i.i.d. Rademacher random variables.
\end{Lemma}

\begin{proof}[Proof of \Cref{lem:con-in-medreg}]
For any $t,\lambda>0$, we have
\begin{eqnarray}
\nonumber \mathbb{P}\left(\sup_{u\in \mathcal{S}^{\dimension-1}}\frac{1}{n}\sum_{i=1}^n\left(\mathbb{E}|X_i^\top u|-|X_i^\top u|\right)>t\right) &\leq& e^{-\lambda t}\mathbb{E}\exp\left(\lambda \sup_{u\in \mathcal{S}^{\dimension-1}}\frac{1}{n}\sum_{i=1}^n\left(\mathbb{E}|X_i^\top u|-|X_i^\top u|\right)\right) \\
\label{eq:lmpf-sym} &\leq& e^{-\lambda t}\mathbb{E}\exp\left(2\lambda \sup_{u\in \mathcal{S}^{\dimension-1}}\frac{1}{n}\sum_{i=1}^n\delta_i|X_i^\top u|\right) \\
\label{eq:lmpf-contra} &\leq& e^{-\lambda t}\mathbb{E}\exp\left(4\lambda  \sup_{u\in \mathcal{S}^{\dimension-1}}\left| \frac{1}{n}\sum_{i=1}^n\delta_i X_i^\top u \right| \right) \\
\label{eq:lmpf-u} &\leq&  e^{-\lambda t}\mathbb{E}\exp\left(8\lambda  \max_{u\in \mathcal{U}}\left| \frac{1}{n}\sum_{i=1}^n\delta_i X_i^\top u \right| \right)\\
\nonumber &\leq& e^{-\lambda t}\sum_{u\in \mathcal{U}}\mathbb{E}\exp\left(8\lambda  \left| \frac{1}{n}\sum_{i=1}^n\delta_i X_i^\top u \right| \right)\\
\nonumber &\leq& 2\exp\left(\frac{32\lambda^2}{n}-\lambda t+\log|\mathcal{U}|\right) \\
\nonumber &\leq& 2\exp\left(-\frac{nt^2}{128}+\dimension\log 6\right).
\end{eqnarray}
The inequality (\ref{eq:lmpf-sym}) is by symmetrization with $\delta_i$'s being independent Rademacher random variables, and the bound \eqref{eq:lmpf-contra} is by \Cref{lem:contraction}. The set $\mathcal{U}$ in (\ref{eq:lmpf-u}) is a $1/2$-cover of $\mathcal{S}^{\dimension-1}$. A standard argument leads to the bound $\sup_{u\in \mathcal{S}^{\dimension-1}}\left| \frac{1}{n}\sum_{i=1}^n\delta_i X_i^\top u \right|\leq 2\max_{u\in \mathcal{U}}\left| \frac{1}{n}\sum_{i=1}^n\delta_i X_i^\top u \right|$ and $|\mathcal{U}|\leq 6^\dimension$. The last inequality above is by taking $\lambda=nt/64$. Finally, the bound will be sufficiently small by taking $t=C\sqrt{\frac{\dimension}{n}}$, which completes the proof of the first inequality.

For the second inequality, we have
\begin{eqnarray*}
&& \mathbb{P}\left(\sup_{u\in \mathcal{S}^{\dimension-1}}\frac{1}{n}\sum_{i=1}^n\left(\mathbb{E}\left(|z_i-rX_i^\top u|-|z_i|\right)-|z_i-rX_i^\top u|+|z_i|\right) > rt\right) \\
&\leq& e^{-\lambda t}\mathbb{E}\exp\left(2\lambda\sup_{u\in \mathcal{S}^{\dimension-1}}\frac{1}{n}\sum_{i=1}^n\delta_i\left(|z_i-rX_i^\top u|-|z_i|\right)/r \right) \\
&\leq& e^{-\lambda t}\mathbb{E}\exp\left(4\lambda  \sup_{u\in \mathcal{S}^{\dimension-1}}\left| \frac{1}{n}\sum_{i=1}^n\delta_i X_i^\top u \right| \right),
\end{eqnarray*}
where we have used symmetrization and \Cref{lem:contraction} with contraction $\psi_i(t_i)=\left(|z_i-rt_i|-|z_i|\right)/r$.
By following the arguments after (\ref{eq:lmpf-contra}), we obtain the desired result.
\end{proof}

\subsection{Proofs of Lower Bound Results}

This section presents proofs of \Cref{lem:matching}, \Cref{lem:tv-gaussian} and \Cref{thm:info-lower}.

\begin{proof}[Proof of \Cref{lem:matching}]
We first assume the equality $\TV(P_1,\cdots,P_m)=\frac{\epsilon}{1-\epsilon}$.
Let $p_1,\cdots,p_m$ be density functions of $P_1,\cdots,P_m$ with respect to some common dominating measure. Define $B_j=\{p_j=\max_{1\leq k\leq m}p_k\}$ for $j=1,\cdots,m$. We set $A_1=B_1$ and $A_j=B_j\backslash (B_1\cup\cdots \cup B_{j-1})$ for $j=2,\cdots,m$. Then, for each $j=1,\cdots,m$, we construct $q_j=\frac{1-\epsilon}{\epsilon}\sum_{k=1}^m(p_k-p_j)\indi_{A_k}$. By its definition, we know that $q_j\geq 0$ and
$$\int q_j= \frac{1-\epsilon}{\epsilon}\left(\sum_{k=1}^m\int_{A_k}p_k-1\right)=\frac{1-\epsilon}{\epsilon}\TV(P_1,\cdots,P_m)=1.$$
This implies $q_j$ is a density. Moreover, since $(1-\epsilon)p_j+\epsilon q_j=\sum_{k=1}^mp_k\indi_{A_k}$ holds for all $j=1,\cdots,m$, the corresponding distributions $Q_1,\cdots,Q_m$ satisfy the desired property.

Now let us consider the more general $\TV(P_1,\cdots,P_m)\leq \frac{\epsilon}{1-\epsilon}$. There exists some $\delta\in[0,\epsilon]$ such that $\TV(P_1,\cdots,P_m)=\frac{\delta}{1-\delta}$. By the previous arguments, there exist $\wt{Q}_1,\cdots,\wt{Q}_m$ such that $(1-\delta)P_j+\delta \wt{Q}_j$ does not depend on $j\in[m]$. Set $Q_j=\frac{\epsilon-\delta}{\epsilon}P_j+\frac{\delta}{\epsilon}\wt{Q}_j$, and we have $(1-\epsilon)P_j+\epsilon Q_j=(1-\delta)P_j+\delta \wt{Q}_j$ does not depend on $j\in[m]$. The proof is complete.
\end{proof}

\begin{proof}[Proof of \Cref{lem:tv-gaussian}]
It suffices to consider $\sigma=1$.
Without loss of generality, assume $\theta_1\leq\cdots\leq\theta_m$. Define
$$A_j=\begin{cases}
\left(-\infty,\frac{\theta_1+\theta_2}{2}\right] & j=1 \\
\left[\frac{\theta_{j-1}+\theta_j}{2},\frac{\theta_{j}+\theta_{j+1}}{2}\right] & j=2,\cdots,m-1 \\
\left[\frac{\theta_{m-1}+\theta_m}{2},\infty\right) & j=m.
\end{cases}$$
We use $p_j$ for the density function of $\mathcal{N}(\theta_j,1)$. Then,
\begin{eqnarray*}
&& \TV\left(\mathcal{N}(\theta_1,1),\cdots,\mathcal{N}(\theta_m,1)\right) \\
&\leq& \sum_{j=1}^m\int_{A_j}p_j-1 \\
&\leq& \frac{1}{2}+ \frac{1}{\sqrt{2\pi}}\frac{\theta_2-\theta_1}{2} + \sum_{j=2}^{m-1}\frac{1}{\sqrt{2\pi}}\frac{\theta_{j+1}-\theta_{j-1}}{2} + \frac{1}{\sqrt{2\pi}}\frac{\theta_m-\theta_{m-1}}{2} + \frac{1}{2} -1 \\
&= & \frac{\theta_m-\theta_1}{\sqrt{2\pi}}.
\end{eqnarray*}
This completes the proof.
\end{proof}

\begin{proof}[Proof of \Cref{thm:info-lower}]
The result follows the construction given in \Cref{sec:lower}. Specifically, we apply Fano's inequality (\ref{eq:fano}) and the Kullback--Leibler divergence bound (\ref{eq:kl-bd-lb}) with $\delta$ set as (\ref{eq:delta-rate-lb-pf}).
\end{proof}

\subsection{Proof of SQ Hardness}\label{sec:pf-SQ}
\newcommand{\anote}[1]{\footnote{{\bf [Ankit: {#1}\bf ]}}}

In this section, we present proofs of the technical results in \Cref{sec:SQ}. We will give formal proofs of \Cref{prop:existence-g-intro}, \Cref{lem:est-implies-testing}, \Cref{fact:Hermite}, and \Cref{lem:chi-squared-bd}, and then apply \Cref{lem:sq-hardness-conditional-ngca} to prove \Cref{thm:ThmFormalSQ}.

\subsubsection{Proof of \Cref{prop:existence-g-intro}}

In this proof, a vector in $\R^{m+1}$ will always be indexed by $[0:m]$. Consider the space $L^\infty:=L^\infty(\R)$ equipped with Lebesgue measure and the weak$^*$-topology $\sigma(L^\infty,L^1)$ using that $(L^1)^* = L^\infty$. We recall a few definitions in the proof sketch.
Let $\cA_i := \{r \in L^\infty: \|r\|_\infty \leq B_i\}$, where $\|\cdot\|_\infty$ is the essential supremum and $B_i=(Cm)^{i/2}$.
Define the linear operator $T:L^\infty \to \R^{m+1}$ that maps $r \in \cA_i$ to $(\int r(x) \phi(x) \he_i(x)dx)_{i=0}^m$.
This is a valid map because $\phi\he_i \in L^1$ and $r \in L^\infty$. Let $\cB_i := \{T(r): r \in \cA_k\} \subset \R^{m+1}$.

\Cref{prop:existence-g-intro} states that the vector $e_i \in \R^{m+1}$ belongs to $\cB_i$.
We shall prove this using LP duality, and for that purpose, we first show that the set $\cB_i$ is compact and convex.

\begin{Lemma} The set $\cB_i$ is compact and convex.
\end{Lemma}
\begin{proof}
The convexity of $\cB_i$ follows directly from the convexity of $\cA_i$ and linearity of $T$.
It remains to show the compactness of $\cB_i$. 
This would follow if $T$ is continuous and $\cA_i$ is compact.
As shown in \cite[Theorem 5.5]{rudin_functional_1996}, $\cA_i$ is weak-$^*$ compact in $L^\infty$ (as an application of  Banach–Alaoglu Theorem) and $T$ is weak-$*$ continuous, implying that $T(\cA_k)$ is compact.
\end{proof}
By the strict hyperplane separation theorem (see, for example, \cite[Section 2.5]{BoydVand04}), which is applicable because $\cB_i$ is both convex and compact,  
if $e_i \not \in \cB_i$, then there must exist a separating hyperplane $u \in \R^{m+1} $ such
that $ \min_{w \in \cB} \langle u, w\rangle >  \langle u,e_i\rangle$.
To establish feasibility of $e_i$, it thus suffices to show that, for any $u \in \R^{m+1}$, there exists $w \in \cB_k$ such that $\langle u,w\rangle \leq \langle u,e_i\rangle$.

For any $u \in \R^{m+1}$, define the associated polynomial $f(x) = \sum_{i=0}^m u_i\he_i(x)$.
Then, for any $r \in \cA_k$ and $w = T(r)$, we have that  $\langle u,w\rangle = \sum_{i=0}^m u_i \int r(x) \phi(x)\he_i(x)dx  = \E[f(G)r(G)]$ with $G\sim\cN(0,1)$.
Moreover, $\langle u,e_i\rangle = \E[f(G)\he_i(G)]$, where we use that Hermite polynomials are orthonormal under Gaussian measure (\Cref{fact:Hermite}).

Therefore, it is equivalent to show that for any polynomial $f$ of degree at most $m$, there exists an $r\in \cA_i$, such that $\E[r(G)f(G)] \leq \E[f(G)\he_i(G)]$.
We shall take $r(x)$ to be $-B_i\mathrm{sign}(f(x))$, which does belong to $\cA_i$, and then the desired inequality becomes: for all polynomials $f$ of degree at most $m$,
it holds that $B_i\E[|f(G)|] \geq \E[f(G)\he_i(G)]$. 
This is proved in the next result:

\begin{Lemma}
\label{lem:sup-poly}
There exists a constant $C > 0$ such that for all $i \in [m]$ for $m \in \N$:
$$\sup_{f: \mathrm{deg(f)}\leq m;\,\, f \neq 0}\frac{\E[f(G)\he_i(G)]}{\E[|f(G)|]} \leq  2\sup_{|y|\leq \sqrt{32m}}|\he_i(y)| \leq ( Cm)^{i/2}.$$
\end{Lemma}
\begin{proof}
The high-level idea is to argue that the expected value of a degree-$O(m)$ polynomial of standard Gaussians should depend primarily on its behavior in a $\Theta(\sqrt{m})$-sized interval around the origin.
Building on this intuition, we shall lower bound the denominator by $\E[|f(G)|\indi_{|x| \leq \sqrt{32m}}]$ and get the following expression:
\begin{align*}
\sup_{f: \mathrm{deg(f)}\leq m}\frac{\E[f(G)\he_i(G)]}{\E[|f(G)|]} &\leq \sup_{f: \mathrm{deg(f)}\leq m}\frac{\E[|f(G)\he_i(G)|]}{\E[|f(G)|\indi_{|G|\leq \sqrt{32m}}]} \\ 
&\leq \sup_{f: \mathrm{deg(f)}\leq m}\frac{\E[|f(G)\he_i(G)|\indi_{|G|\leq \sqrt{32m}}]}{\E[|f(G)|\indi_{|G|\leq \sqrt{32m}}]} \cdot \sup_{f': \mathrm{deg(f')}\leq 2m}\frac{\E[|f'(G)|]}{\E[|f'(G)|\indi_{|G|\leq \sqrt{32m}}]}\\
&\leq \sup_{y: |y|\leq \sqrt{32m}}|\he_i(y)| \cdot \sup_{f': \mathrm{deg(f')}\leq 2m}\frac{\E[|f'(G)|]}{\E[|f'(G)|\indi_{|G|\leq \sqrt{32m}}]}.
\end{align*}
We shall now formalize this intuition that the second term above can be upper bounded by a constant, in fact, $2$.
That is, for any polynomial $f'$ of degree at most $2m$,
$\E[|f'(G)|\indi_{|G| \geq \sqrt{32m}}] \leq 0.5\E[|f'(G)|]$.

Applying Cauchy-Schwarz and Gaussian concentration\footnote{That is, $\P(|G|\geq t) \leq e^{-t^2/2}$ for any $t\geq 1$; see, for example, \cite[Proposition 2.1.2]{Vershynin18}.},
we get that
\begin{align}
\nonumber \E[|f'(G)|\indi_{|G| \geq \sqrt{32m}}] &\leq \sqrt{\E[f'^2(G)]} \sqrt{\mathbb{P}(|G|\geq \sqrt{32m})} \leq \sqrt{\E[f'^2(G)]} \sqrt{e^{-32m/2}} \\ 
&= \sqrt{\E[f'^2(G)]} e^{-8m}.
\label{eq:gausisan-conc-0}
\end{align}
Unfortunately, this is not quite the result we wanted: $\E[|f'(G)|]$ has been replaced with $\sqrt{\E[f'^2(G)]}$.
These two can be related using hypercontractivity of Gaussian polynomials.
We shall use the following consequence of Bonami lemma and Paley-Zygmund inequality:
\begin{Fact}[{\cite[Corollary D.3]{KunWB19}}]
Let $f':\R \to \R$ be a polynomial with  $\mathrm{deg}(f') \leq 2m$.
Then $\bbP(|f'(G)| > 0.5\sqrt{\E[f'^2(G)]}) \geq 0.25 e^{-2m}$.
\end{Fact}
As a direct consequence, we get that $\E[|f'(G)|] \geq  \sqrt{\E[f'^2(G)]} \exp(-2m)/8$.
Combining this with (\ref{eq:gausisan-conc-0}), we get 
\begin{align}
\label{eq:gausisan-conc-1}
\E[|f'(G)|\indi_{|G| \geq \sqrt{32m}}] \leq8  e^{2m} e^{-8m} \E[|f'(G)|]\leq \frac{1}{2}\E[|f'(G)|].
\end{align}
Finally, we use $|\he_i(x)| \leq (C(1 + |x|))^i$  by \Cref{fact:Hermite}.
\end{proof}

\subsubsection{Proofs of \Cref{lem:est-implies-testing}, \Cref{fact:Hermite} and \Cref{lem:chi-squared-bd}}

\begin{proof}[Proof of \Cref{lem:est-implies-testing}]
Consider an SQ estimation algorithm $\cA (\STAT_{P,\tau})$ and we define the test $\widehat{\theta} = \indi\{\|\cA (\STAT_{P,\tau})\|_2 > \delta/4\}$. Under the condition that $\sigma\in[1/2,1]$ and $R=(1 - \epsilon) \cN(0, \sigma^2) + \epsilon D$ for some $D$, we know that $P\in\bigcup_{(\beta,\sigma): \|\beta\|_2 \leq 1, \sigma\in[1/2,1]}\mathcal{D}_{\beta,\sigma,\epsilon}$ for both $P$ under null (\ref{eq:reg-test-h0}) and alternative (\ref{eq:reg-test-h1}), which implies $\bbP\left(\bigl\|\cA (\STAT_{P,\tau}) -\beta \bigr\|_2 \leq  \sigma\delta/4\right) > 0.9$ for these $P$'s. For $P$ under null (\ref{eq:reg-test-h0}), we have $\|\beta\|_2=0$ and $\sigma=1$, and thus $\bigl\|\cA (\STAT_{P,\tau}) -\beta \bigr\|_2 \leq  \sigma\delta/4$ implies $\wh{\theta}=0$. For $P$ under alternative (\ref{eq:reg-test-h1}), we have $\|\beta\|_2=\delta$ and $\sigma\in[1/2,1]$, and thus $\bigl\|\cA (\STAT_{P,\tau}) -\beta \bigr\|_2 \leq  \sigma\delta/4$ implies $\wh{\theta}=1$ by triangle inequality. Hence, $\widehat{\theta} = \indi\{\|\cA (\STAT_{P,\tau})\|_2 > \delta/4\}$ is an SQ test solving the testing problem (\ref{eq:reg-test-h0})-(\ref{eq:reg-test-h1}).
\end{proof}
\begin{proof}[Proof of \Cref{fact:Hermite}]
The first three facts are standard and we refer the reader to \cite{Szego89}.
The final property also follows by standard properties of Hermite polynomials as follows: The Hermite polynomials have the following expansion $\he_{i}(x) = \sqrt{i!}\sum_{m=0}^{\lfloor i/2\rfloor}\frac{(-1)^m}{m! (i-2m)!}  \frac{x^{i-2m}}{2^m}$ and therefore $|\he_i(x)| \leq   \sum_{m=0}^{\lfloor i/2\rfloor} \frac{(2m!)}{\sqrt{i!} 2^mm!} {i \choose 2m}x^{2m} \lesssim (1 + |x|)^i$ provided $\frac{(2m!)}{2^m\sqrt{i!} m!}\lesssim1$. To see the latter, observe that  the maximum over $m$ is achieved when $2m = i$ and the expression then becomes $\frac{\sqrt{i!}}{2^{i/2}((i/2)!)}$, which by Stirling's approximation is at most  $i^{1/4}\frac{(i/e)^{i/2}}{ 2^{i/2} \sqrt{i/2} (i/(2e))^{i/2}}\asymp i^{-1/4}\lesssim1$.
\end{proof}

\begin{proof}[Proof of \Cref{lem:chi-squared-bd}]
Here, we calculate the $\E_{y \sim R}[\chi^2(A_y, \cN(0,1))]$.
First, let us define $H_y(x) = A_y(x) R(y)$ so as to change the base measure from $y \sim R$ to $y \sim \cN(0,1)$ through the following series of simple inequalities:
\begin{align*}
\E_{y \sim R}\left[\chi^2(A_y, \cN(0,1))\right] + 1 
&=  \int \left(\int \frac{A_y(x)^2}{\phi(x)} dx\right) R(y)dy \\ 
&=  \int \left(\int \frac{H_y(x)^2}{R(y) \phi(x)} dx\right) dy \\ 
&\leq  \frac{1}{(1- \epsilon)} \int \left(\int \frac{H_y(x)^2}{\phi(y) \phi(x)} dx\right) dy \\ 
&= \frac{1}{(1-\epsilon)} \E_{y \sim \cN(0,1)} \left[\int \frac{H_y(x)^2}{\phi^2(y) \phi(x)} dx\right], 
\numberthis 
\label{eq:chi-sq-1}
\end{align*}
where the inequality follows by noting that $R(y) =(1-\epsilon)\phi(y) + \epsilon D(y) \geq (1 - \epsilon) \phi(y)$. 
Using the definition of $A_y$ in (\ref{eq:Ay-alt-form}), we get that
\begin{align*}
\frac{|H_y(x)|}{\phi(y)} &\leq \phi(x;\delta y; 1 -\delta^2) + \epsilon \frac{D(y)}{\phi(y)} D_y(x) \leq \phi(x;\delta y; 1 -\delta^2) + \epsilon \frac{D(y)}{\phi(y)}\phi(x) + \epsilon \frac{D(y)}{\phi(y)}|f_y(x)|,
\end{align*}
where the last inequality uses (\ref{eq:D=phi+f}).
Applying approximate triangle inequality,
we get that the integral is upper bounded as  follows:
\begin{align}
\label{eq:chi-square-2}
\int \frac{H_y(x)^2}{\phi^2(y) \phi(x)} dx \lesssim \int \frac{\phi^2(x;\delta y; 1 -\delta^2)}{\phi(x)}dx + \epsilon^2 \left(\frac{D(y)}{\phi(y)}\right)^2 + \epsilon^2 \int_x \left(\frac{D(y)}{\phi(y)} \frac{f_y(x)}{\phi(x)}\right)^2 \phi(x) dx.  
\end{align}
The first term above equals  $\frac{\exp\left({(\delta^2 y^2)/{(1 + \delta^2)}}\right)}{\sqrt{1 - \delta^4}}$.
To control the second term, we will compute an upper bound on $\frac{D(y)}{\phi(y)}$. Using (\ref{eq:qprime-def}), we get that  $ \frac{D(y)}{\phi(y)} \leq 1+  \sum_{i \in [m]} |a_i(y)|\sup_{x}|g_i(x)|/\phi(x) \leq 1 + \sum_{i \in [m]}|a_i(y)|B_i$ with $B_i=(Cm)^{i/2}$.
A similar upper bound exists for the third term as well:
\begin{align*}
\frac{D(y)}{\phi(y)} \frac{|f_y(x)|}{\phi(x)} \leq  \sum_{i=1}^m  |a_i(y)| \frac{|g_i(x)|}{\phi(x)}  \leq  \sum_{i=1}^m  |a_i(y)| B_i\,.
\end{align*}
Plugging these bounds back to (\ref{eq:chi-square-2}), we obtain that the desired integral is at most
\begin{align*}
\int \frac{H_y(x)^2}{\phi^2(y) \phi(x)} dx  \lesssim 1 + \frac{e^{\frac{\delta^2 y^2 }{1 + \delta^2}}}{\sqrt{1 - \delta^4}} + \left(\sum_{i=1}^m|a_i(y)|B_i\right)^2,
\end{align*}
and the desired expression in (\ref{eq:chi-sq-1}) becomes
\begin{align*}
\E_{y \sim R}\left[\chi^2(A_y, \cN(0,1))\right] &\lesssim 1 + \E_{G \sim \cN(0,1)}\left[e^{\frac{\delta^2 G^2 }{1 + \delta^2}}\right] + \E_{G \sim \cN(0,1)}\left[\left(\sum_{i=1}^m|a_i(G)|B_i\right)^2\right] \\ 
&\lesssim 1 + m \sum_{i=1}^m \E_{G \sim \cN(0,1)}\left[ a_i^2(G) B_i^2 \right]\lesssim 1 + m \sum_{i=1}^m \E_{G \sim \cN(0,1)}\left[ \frac{\delta^{2i}}{\epsilon^2} \he_i^2(G) B_i^2 \right]\\
&= 1 + m \sum_{i=1}^m  \frac{\delta^{2i}}{\epsilon^2} B_i^2  
\leq 1 + m \sum_{i=1}^m \left( C m \delta^2/\epsilon^2  \right)^i \lesssim m,
\end{align*}
where we use that $ Cm \delta^2/\epsilon^2 \leq 1/4$.
\end{proof}

\subsubsection{Proof of \Cref{thm:ThmFormalSQ}}





Recall that we set $\sigma^2=1-\delta^2$ and $Q_X=E_{v^\top X}$ in (\ref{eq:reg-test-h1}), where the conditional distribution $E_{v^\top X}$ is determined by (\ref{eq:diff-factor-contam}) with $D(\cdot)$ and $D_y(\cdot)$ constructed according to (\ref{eq:qprime-def}), (\ref{eq:D=phi+f}) and (\ref{eq:fy-final-form}). This leads to $R$ and $A_y$  given by (\ref{eq:ngca-R}) and (\ref{eq:ngca-Ayx'}) such that the testing problem (\ref{eq:reg-test-h0})-(\ref{eq:reg-test-h1}) is identical to (\ref{eq:con-ngca-test0})-(\ref{eq:con-ngca-test1}), an instance of the conditional NGCA. \Cref{prop:existence-g-intro} guarantees that the construction is valid and $A_y$ matches $m$ moments with $\cN(0,1)$ for all $y\in\R$ as long as $m\leq \frac{\epsilon^2}{4C\delta^2}$. We thus take $m=\floor{\frac{\epsilon^2}{4C \delta^2}}$, and the result will follow from \Cref{lem:sq-hardness-conditional-ngca} after some simplifications.

First, we compute the lower bound on  the number of queries.
\Cref{lem:sq-hardness-conditional-ngca} yields a lower bound of $\ell/\dimension^{(m+1)/4}$ for $\ell = 2^{\Omega(\dimension^{\Omega(1)})}$.
When the condition $\dimension \gtrsim \left( 2m \log \dimension\right)^{\Omega(1)}$ for a large implicit constant $\Omega(1)$ holds, we have that  that $\dimension^{(m+1)/4} \leq \sqrt{\ell}$, and thus the lower bound on the number of queries is at least $\ell/\sqrt{\ell}=\sqrt{\ell} = 2^{\Omega(\dimension^{\Omega(1)})}$.
Since $m \lesssim \epsilon^2/ \delta^2$, this conditioned is satisfied when $\dimension \gtrsim ( \epsilon/\delta)^{\Omega(1)}$, as in the statement of \Cref{thm:ThmFormalSQ}.

Next, we compute an upper bound on the tolerance.
As shown in \Cref{lem:chi-squared-bd},
the construction satisfies that $\E_{y \sim R}[\chi^2(A_y,\cN(0,1))] \lesssim m \lesssim \epsilon^2/\delta^2$.
Therefore, the upper bound on tolerance from \Cref{lem:sq-hardness-conditional-ngca} is $O(\frac{\sqrt{m}}{\dimension^{(m+1)/8}})$,
which is at most $O(\frac{\sqrt{m}}{\dimension^{m/16}})$ when $m  \geq 1$, which holds since $\epsilon/\delta\gtrsim 1$.
Observe that this can be further upper bounded by $O(\frac{1}{\dimension^{m/32}})$, when $m$ is at least a large enough constant\footnote{Indeed, it suffices to establish that $\sqrt{m} \leq 3^{m/32}$, which holds as long as $m$ is at least a large enough constant say $10^4$.}, which again reduces to requiring that  $\epsilon/\delta\gtrsim 1$.
Finally, the hardness of the estimation problem is implied by the hardness of the testing problem as argued in \Cref{lem:est-implies-testing}; see the discussion below \eqref{eq:ngca-R} which shows that $R$ satisfies the assumption of \Cref{lem:est-implies-testing}.


\subsection{Proofs of Results in \Cref{sec:disc}}

This section presents proofs of \Cref{prop:different-models} and \Cref{thm:tail}.

\subsubsection{Proof of \Cref{prop:different-models}}

We establish these one-by-one.
\begin{enumerate}[label=(\roman*)]
\item This is immediate since \modelref{model:adaptive} permits an arbitrary contaminating conditional distribution $Q_X$, and both \modelref{model:obv-1} and \modelref{model:obv-2} are special cases.
  \item We need to compute the TV distance between $P_{\beta,\sigma}$ and $M$, where $M$ is the joint distribution in \Cref{model:non-uni}.
Since the marginals of $X$ is identical between $P_{\beta,\sigma}$ and $M$,
we have that $\TV(P_{\beta,\sigma},M) = \E_{X \sim \cN(0,I_p)}[\TV(P_{\beta,\sigma}(y \mid X), M(y \mid X))]$, which is at most $\E[ \epsilon_X ] \leq \epsilon$. The realtionship with \Cref{model:adaptive} is trivial.

\item The relationship with \Cref{model:huber} is straightforward.
To compare the relationship with \Cref{model:adaptive}, consider the special case where $\epsilon Q(x) = \epsilon \phi(x)$.
That is, $Q(x,y) = Q(x) Q(y \mid x) = \phi(x) Q(y \mid x)$ and observe that this satisfies \Cref{model:adaptive}.

\item The upper bound follows by the minimax rate under \modelref{model:TV} as per \cite{Gao20}.\footnote{The upper bound in \cite{Gao20} is established under Huber contamination. However, a slight modification of the proof also works under TV contamination.}
The lower bound of $\sigma \sqrt{\frac{p}{n}}$ follows from the special case of $\epsilon = 0$.
Thus, it suffices to exhibit a lower bound of $\Omega(\epsilon)$, for which we construct two different lower bound instances. 

\paragraph{Lower bound for \Cref{model:non-uni}.} 
Consider two different linear models $P := P_{\beta,\sigma}$ and $P':=P_{-\beta,\sigma}$.
Define the distribution $M$ where $X \sim \cN(0,I_p)$
and the conditional pdf of $y \mid x$ is $M(y \mid x) = \frac{\max(P(y \mid x),P'(y \mid x))}{\int \max(P(y \mid x),P'(y \mid x)) dy} = \frac{\max(P(y\mid x),P'(y \mid x))}{1 + \TV(P(y \mid x),P'(y \mid x))}$.
For each $x$, we set 
\begin{align*}
\epsilon_x = \frac{\TV(P(y \mid x),P'(y \mid x))}{1 + \TV(P(y \mid x),P'(y  \mid x))} \text{ and  } Q_x = \frac{1}{\epsilon_x}\frac{\max(0,P'(y \mid x) - P(y \mid x) )}{1 + \TV(P(y \mid x),P'(y \mid x))}.
\end{align*}
Then, 
 it can be seen that 
$$(1 - \epsilon_x) P(y \mid x) + \epsilon_x Q_x = M(y \mid x)\,.$$
Furthermore, $Q_x$ is a valid distribution because the right hand side is non-negative and has integral over $y$ equal to $\frac{1}{\epsilon_x} \frac{\TV(P(y \mid x),P'(y \mid x))}{1 + \TV(P(y \mid x),P'(y \mid x))} = 1$.
 An analogous calculation shows that there exist $Q'_x$ such that 
$ (1 - \epsilon_x) P'(y \mid x) + \epsilon_x Q_x' = M(y \mid x)\,$.

To finish the lower bound of $\Omega(\sigma \epsilon)$, it thus suffices to show that $\E[\epsilon_x]\lesssim \|\beta\|_2/\sigma $.
Plugging in the expression aobve, we have that $\E[\epsilon_X] \leq \E[\TV(P(y \mid X),P'(y \mid X))]
= \E[\TV(\cN(\beta^\top X, \sigma^2),\cN( - \beta^\top X, \sigma^2))] \leq \E[|2 \beta^\top X|/\sigma] \lesssim \|\beta\|_2/\sigma$. This concludes the proof for \Cref{model:non-uni}.


\paragraph{Lower bound for \Cref{model:hub-bdd-marg}.}
The lower bound instance for the Huber contamination model would be applicable here. For the same $P$ and $P'$ defined above with some $\beta$, define
$\epsilon = \frac{\TV(P, P')}{1 + \TV(P,P')}$.
We take $Q(x,y)$ to be 
\begin{align*}
Q(x,y)  = \frac{\max(0, P'(x,y) - P(x,y))}{\TV(P,P')} \text { and } Q'(x,y)  = \frac{\max(0, P(x,y) - P'(x,y))}{\TV(P,P')}.
\end{align*}
It can then be verified that $(1- \epsilon)P + \epsilon Q=(1 -\epsilon)P' + \epsilon Q'$
and that both $Q$ and $Q'$ are valid joint distributions over $X$ and $y$.
Finally, 
the marginal of $Q$ is
\begin{align*}
Q(x) &= \int Q(x,y) dy = \frac{ \phi(x) }{\TV(P,P')} \int \max(0, P'(y|x) - P(y|x))dy \\
&= \frac{\phi(x)}{\TV(P,P')} \TV(\cN(\beta^\top X, \sigma^2), \cN(-\beta^\top X, \sigma^2))\,.
\end{align*}
In particular,
$\epsilon Q(x)/\phi(x) \leq \frac{\epsilon}{\TV(P,P')} = \frac{1}{1 + \TV(P,P')} \leq 1$, as desired. Similarly, we also have $\epsilon Q'(x)/\phi(x)\leq 1$, which concludes the proof.
\end{enumerate}

\subsubsection{Proof of \Cref{thm:tail}}

The upper bound follows the same arguments in the proof of \Cref{thm:info-upper}. We obtain the first inequality in the display (\ref{eq:usedlatertail}) with $\Phi$ and $\phi$ replaced by $F$ and $f$. That is,
$$\frac{\|\wh{\beta}-\beta\|_2}{\sigma}\wedge \frac{1}{2t}\leq \frac{4\epsilon(1-F(t)) + 2C\sqrt{\frac{\dimension}{n}}}{tf(1)\mathbb{P}\left(2t>|Z|\geq t\right)},$$
where $Z$ has density $f(t)\propto e^{-|t|^{\gamma}/2}$ and CDF $F(t)=\int_{-\infty}^tf(x)dx$. Similar to the Gaussian case, we have $\frac{(1-F(t))}{\mathbb{P}\left(2t>|Z|\geq t\right)}\leq C_2$ and $\mathbb{P}\left(2t>|Z|\geq t\right)\geq C_3e^{-|t|^{\gamma}}$ for all $t>0$, which leads to the bound
$$\frac{\|\wh{\beta}-\beta\|_2}{\sigma}\wedge \frac{1}{2t}\leq \frac{C_1}{t}\left(\epsilon+\sqrt{\frac{\dimension}{n}}e^{|t|^{\gamma}}\right).$$
We therefore obtain the desired bound when $\sqrt{\frac{\dimension}{n}}+\epsilon<c$ by taking $t=\left(\log(n\epsilon^2/\dimension+e)/3\right)^{1/\gamma}$.

The lower bound follows the arguments in \Cref{sec:lower}. In particular, we need to set $\delta$ so that for some constant $C_4>0$, the inequality $\left(2\delta t/\sigma-\sqrt{\pi/2}\epsilon\right)_+^2e^{-|t|^{\gamma}/4}\leq \frac{C_4\dimension}{n}$ holds for all $t>0$. Rearranging this inequality leads to the choice
$$\delta=\min_{t>0}\left[\frac{\sigma}{2t}\left(\sqrt{\frac{\pi}{2}}\epsilon + \sqrt{\frac{C_4\dimension}{n}}e^{|t|^{\gamma}/8}\right)\right]\asymp \sigma\left(\sqrt{\frac{\dimension}{n}}+\frac{\epsilon}{\left(\log(n\epsilon^2/\dimension+e)\right)^{1/\gamma}}\right),$$
which is the desired rate.

\printbibliography
\end{sloppypar}

\end{document}